\documentclass[reqno]{article}
\pdfoutput=1

\usepackage{lipsum}
\usepackage{amsfonts}
\usepackage{graphicx}
\usepackage{epstopdf}
\usepackage{algorithmic}
\usepackage{fancyhdr}
\usepackage{url}
\ifpdf
  \DeclareGraphicsExtensions{.eps,.pdf,.png,.jpg}
\else
  \DeclareGraphicsExtensions{.eps}
\fi

\usepackage{graphicx}
\usepackage{amsmath}
\usepackage{bm}

\font\Bbb=msbm10 scaled 1000
\font\sBbb=msbm7 scaled 1000
\font\gothic=eufm10 scaled 1100
\font\smallgothic=eufm7 scaled 1100

\def\bb#1{\mbox{\sBbb #1}}             
\def\BB#1{\mbox{\Bbb #1}}              
\def\CC#1{\mathrm{#1}}                 
\def\Gg#1{\mbox{\smallgothic #1}}      
\def\GG#1{\mbox{\gothic #1}}           
\def\Mm#1{\mbox{\boldmath$\scriptstyle #1$\unboldmath}}
\def\MM#1{\mbox{\boldmath$#1$\unboldmath}}

\def\R#1{(\ref{#1})}

\def\D{\,\CC{d}}

\newenvironment{case}
{\left\{ \arraycolsep=12pt
  \hspace*{-9pt}
  \begin{array}{ll}}
  {\end{array}
    \right.
    \hspace*{-9pt}}

\newcommand{\ud}{\mbox{d}}

\usepackage{amsopn}

\newtheorem{theorem}{Theorem}

\newcommand{\QED}{\hspace*{2em}\hfill$\Box$}
\newenvironment{proof}[1]{\vspace{-2pt}{\em Proof\/}\quad #1}{\QED
  \vspace{8pt}}

\newcommand{\ee}{{\mathrm e}}

\newcommand{\hyper}[5]
       {{}_{#1} F_{#2} \!\left[
           \begin{array}{l}
         #3;\\#4;
           \end{array}#5\right]
       }

\newcommand{\AAA}{\mathcal{A}}
\newcommand{\DI}{\circle*{4}}
\newcommand{\UP}{\line(0,1){12}}
\newcommand{\LE}{\line(-1,1){12}}
\newcommand{\RI}{\line(1,1){12}}

\title{Magnus expansions and  pseudospectra of ~\\ Master Equations}

\author{Arieh Iserles\thanks{Department of Applied Mathematics and Mathematical Physics,
University of Cambridge, Wilberforce Road, Cambridge CB3 0WA, UK ({\tt a.iserles@damtp.cam.ac.uk}).}
        \and Shev MacNamara\thanks{The School of Mathematics and Statistics, University of New South Wales, UNSW, Sydney, Australia {\tt (s.macnamara@unsw.edu.au)}}}

\begin{document}

\maketitle

\begin{abstract}
 New directions in research on master equations are showcased by example.
 Magnus expansions, time-varying rates, and pseudospectra are highlighted.
 Exact eigenvalues are found  and contrasted with the large errors produced by standard numerical methods in some cases.
Isomerisation  provides a running example and an illustrative application to chemical kinetics.
  We also give a brief example of the totally asymmetric exclusion process.
\end{abstract}


\textbf{AMS}:   60J28, 60H35, 65F40, 65F99, 65L15 , 65FL99, 92C40

\section{Introduction}

The term `master equation' goes back at least as far as the work of Kac  in the middle of the twentieth century \cite[page 105]{Kac1957bookMasterEquation}, and the subject of master equations admits a Feynman--Kac stochastic path integral formulation \cite{MarkusWeber2016reviewFeynmanKacMasterEqn}. 
The general principle of a governing equation emerging from ensemble averages goes back much further in the history of statistical mechanics, including the kinetic theories of Boltzmann and, earlier, of Bernoulli in the 1700s.
Generalised master equations can cater to some form of memory and therefore be non-Markovian but the most common interpretation of master equations is as Markov processes.
Perhaps the first application of the eponymous Markov process was Andrei Markov's model of a poem, ``Eugeny Onegin,'' as a Markov chain, which he presented in 1913 in St Petersburg.
Other famous applications include Shannon's Information Theory and Google's PageRank to find order in the information on the World Wide Web  \cite{Top5MarkovApplications}.
Choosing the simplest examples, we describe applications to exclusion processes and chemical processes, although the computational methods we present have wider applicability.


\subsection{Models of isomerisation}
The same chemical species can sometimes exist in two distinct molecular forms, $S_1$ and $S_2$, and can reversibly convert from one form, or isomer, to the other in a process named \textit{isomerisation}:
$
S_1 \longleftrightarrow S_2.
$
A mathematical model involves two rate constants (this terminology is common, but in our examples the rate `constants' are often \textit{time-dependent}), $c_1(t)$ associated with the forward reaction
$
S_1 \stackrel{c_1}{\longrightarrow} S_2 ,
$
and $c_2(t)$ for the backward reaction
$
S_1 \stackrel{c_2}{\longleftarrow} S_2.
$

A hierarchy of three mathematical frameworks for modelling chemical reactions is provided by the reaction rate equations (RRE), the chemical Langevin equation, and the chemical master equation (CME).
Typically when all species are present in high concentrations, the deterministic reaction rate equations are a good model at a macroscopic scale, but if some species are present in small numbers of molecules then often the discrete and stochastic CME is a more appropriate model at a mesoscopic scale \cite{MacBur08, PetzoldCloud, PavliotisStuart2008book}. 
Stochastic differential equations such as the Langevin equation for isomerisation  \cite{Gillespieisomerization2002} and their corresponding Fokker--Planck partial differential equations provide models at scales that are intermediate between those of the deterministic rate equations and the discrete and stochastic master equations.



The reaction rate equations for this model of isomerisation are the two ordinary differential equations (ODEs)
\begin{equation}
\frac{\ud}{\ud t}[S_1] = -c_1(t) [S_1] + c_2(t) [S_2], \qquad
\frac{\ud}{\ud t}[S_2] = +c_1(t) [S_1] - c_2(t) [S_2],
\label{eq:0}
\end{equation}
where $[S_i]$ indicates the concentration (molecules per unit volume) of species $i$.

The master equation for this model is  a continuous time,  discrete state Markov process for which  a linear system of ODEs, $p' = A p$, describes the evolution of the associated probability distribution $p$.
The $i$th state records the integer number of molecules of each species, and the probability of this state is recorded in the $i$th entry of the vector $p$.
In a small time $\ud t$, the probability mass that flows  from state $j$ to a different state $i$ is approximately given by $A_{ij} \ud t$.
The matrix $A$ has nonnegative off-diagonals and zero column sum, and is thus a \textit{graph Laplacian}.
As an example, if we start with $N$ molecules of species $S_1$ and zero molecules of  $S_2$, then there are $N+1$ states, $(i,N-i)$ for $i=0,\ldots,N$, where state $i$ has $i$ molecules of $S_1$.
If our initial condition has all probability concentrated on state $(0,N)$, then our initial probability vector is $p(0) = (0, 0, \ldots, 1)^\top$.
With rates $c_1(t) = 1+ f(t)$ and $c_2(t) = 1-f(t)$,  the probability vector evolves according to the linear ODE \eqref{eq:1.1}, introduced below, which is the CME for isomerisation.

``Generally, the CME has such extremely high dimension that it cannot be handled analytically or computationally'' \cite{HighamD2008}.
In this article we focus on some exceptions.
A large class of important and solvable models, including isomerisation, arise when reaction rates are linear as a function of the state \cite{Jahnke2007}.
For this special class of models we have exact agreement between the average value of the stochastic CME model and the solution of the corresponding deterministic reaction rate equations.
(Usually these models agree only approximately.)
The exact solution to the CME \eqref{eq:1.1} for our isomerisation example is a binomial distribution, where the time-varying parameter in the binomial distribution comes from the solution to the corresponding RRE \eqref{eq:0}.
This makes it an ideal candidate for demonstrating novel applications of Magnus methods, which as we will see, reveal finer structure in the master equations.

\subsection{A master equation for isomerisation with explicitly time-varying rates}
We are concerned with the linear ODE
\begin{equation}
  \label{eq:1.1}
\frac{\ud}{\ud t}  \MM{p}= \left[ A^{[0]}+A^{[1]} f(t) \right]\! \MM{p},\qquad \qquad \MM{p}(0)=\MM{p}_0\in\BB{R}^{N+1},
\end{equation}
involving two matrices $A^{[0]}$ and $A^{[1]}$ defined by, for $k,\ell=0,\ldots,N$,
\begin{equation}
\label{eq:A0:A1:definitions}
  A^{[0]}_{k,\ell}=
  \begin{case}
    -N, & k=\ell,\\
    \ell, & k=\ell-1,\\
    N-\ell, & k=\ell+1,\\
    0, & \mbox{otherwise;}
  \end{case}\qquad A^{[1]}_{k,\ell}=
  \begin{case}
    N-2\ell, & k=\ell,\\
    \ell, & k=\ell-1,\\
    -N+\ell, & k=\ell+1,\\
    0, & \mbox{otherwise.}
  \end{case}
\end{equation}
The $A^{[0]}$ matrix is remarkably close to the \texttt{`clement'} matrix in the \texttt{MATLAB} gallery, which has a zero main diagonal but is otherwise the same.

If $-1 \le f(t) \le 1$ then $\AAA = A^{[0]}+A^{[1]} f(t) $ has the usual properties of a graph Laplacian matrix (sometimes called the \textit{infinitesimal generator} of the Markov process).
In that case (\ref{eq:1.1}) is a master equation, which was originally simulated  for the special case $f(t)=\sin t$ \cite{KormannMacNamara2016}. 
Here, we generalize.
It turns out  \R{eq:1.1} has a truly miraculous structure.

\section{The Magnus expansion}
\label{sec:magnus:expansion}
The matrix exponential is essentially the solution of a linear ODE when the coefficient matrix is constant, i.e.
\begin{equation}
 \frac{\ud}{\ud t}  \bm{p} = \mathbb{A} \bm{p}  \qquad \textrm{\;\; with solution \, } \quad \bm{p}(t) = \exp(t \mathbb{A} ) \bm{p}(0).
\label{eq:constant:matrix:ODE}
\end{equation}
When the matrix varies in time, $\mathbb{A}=\mathbb{A}(t)$, the solution is no longer simply the matrix exponential, but it can still be expressed in an exponential form.
We write
\begin{equation}
 \frac{\ud}{\ud t}  \bm{p} = \mathbb{A}(t) \bm{p}  \qquad \textrm{\;\; with solution \, } \quad \bm{p}(t) = \exp( \MM{\Omega}(t) ) \bm{p}(0).
\label{eq:ODE:magnus:form}
\end{equation}
Here, the Magnus expansion \cite{Magnus54} tells us how to find the crucial matrix $\MM{\Omega}(t)$ as an infinite series, namely
\begin{equation}
	\MM{\Omega} (t) = \int_{0}^{t} \mathbb{A}(s) \ud s - \frac{1}{2} \int_{0}^{t} \left[\int_{0}^{s}\mathbb{A}(r) \ud r, \mathbb{A}(s)\right] \ud s + \ldots.
\label{eq:ODE:magnus:first:two:terms}
\end{equation}
All higher order terms in the expansion can be generated recursively by integration and commutation, thus involving commutators as a factor.
The \textit{commutator}  of two matrices is, as usual, $[A,B] \equiv AB -BA$.
In the special case that the matrix commutes with itself for all time, i.e. $[\mathbb{A}(t_1), \mathbb{A}(t_2)] \equiv 0$, those commutators are all zero so the expansion simplifies to $\Omega (t) = \int_0^{t} \mathbb{A}(s) \ud s $, agreeing with our intuition from the scalar case.
This expansion, which is valid for all sufficiently small times $t$, was originally motivated by applications in quantum mechanics where it was derived by an analogy with Cauchy--Picard iteration in the 1950s.
For a long time it remained merely a theoretical tool, and it was only nearing the turn of the century that it was fashioned into an effective computational tool  \cite{iserles00lgm}.

A remarkable correspondence between terms in the Magnus expansion and rooted, binary trees (elucidated in \cite[equation (4.10)]{iserles00lgm}) allows \eqref{eq:ODE:magnus:first:two:terms} to be written as
\begin{equation}
  \label{eq:2.1}
  \Omega(t)=\sum_{m=0}^\infty \sum_{\tau\in\bb{T}_m} \int_0^t \alpha(\tau) G_\tau(x)\D x .
\end{equation}
All terms in the expansion are identified with a rooted, binary tree in the set of Magnus trees, denoted $\cup_m \BB{T}_m$.
In this correspondence vertical lines correspond to integration and joining trees corresponds to commutation.
Here is the four-step recipe.
\begin{enumerate}
\item $\BB{T}_m$ is the set of Magnus trees with $m$ vertical lines.
\item The only member of $\BB{T}_0$ is \begin{picture}(10,10) \put (5,5) \DI \end{picture}.
\item $\tau\rightarrow G_\tau$ is a mapping from Magnus trees to matrices. Specifically,  $G_\bullet=\AAA$ and, given $m\geq1$, any $\tau\in\BB{T}_m$ can be represented in the form
\begin{equation}
  \label{eq:2.2}
  \tau=
  \begin{picture}(24,26)
    \put (12,0) \DI
    \put (12,0) \LE
    \put (12,0) \RI
    \put (0,12) \DI
    \put (0,12) \UP
    \put (-4,28) {$\tau_1$}
    \put (20,16) {$\tau_2$}
  \end{picture},\qquad \tau_1\in\BB{T}_{m_1},\; \tau_2\in\BB{T}_{m_2},\quad m_1+m_2=m-1.
\end{equation}
In that case
\begin{displaymath}
  G_\tau(t)=\left[\int_0^t G_{\tau_1}(x)\D x,G_{\tau_2}(t)\right]\!.
\end{displaymath}
\item $\alpha:\tau\rightarrow\BB{Q}$ is a mapping from Magnus trees to rational numbers.
Specifically, $\alpha(\bullet)=1$ and, for any $\tau\in\BB{T}_m$ for $m\geq1$, with $\CC{B}_s$ denoting Bernoulli numbers,
\begin{displaymath}
  \tau=
  \begin{picture}(60,64)
    \put (12,0) \DI
    \put (12,0) \LE
    \put (12,0) \RI
    \put (0,12) \DI
    \put (0,12) \UP
    \put (-4,28) {$\eta_1$}
    \put (24,12) \DI
    \put (24,12) \LE
    \put (24,12) \RI
    \put (12,24) \DI
    \put (12,24) \UP
    \put (8,40) {$\eta_2$}
    \put (48,36) \DI
    \put (48,36) \LE
    \put (48,36) \RI
    \put (36,48) \DI
    \put (60,48) \DI
    \put (36,48) \UP
    \put (32,64) {$\eta_s$}
    \multiput (38,26)(2,2){4}{\circle*{1}}
  \end{picture}\qquad \Rightarrow\qquad \alpha(\tau)=\frac{\CC{B}_s}{s!} \prod_{j=1}^s \alpha(\eta_j).
\end{displaymath}
\end{enumerate}

In general, this procedure elegantly expresses the Magnus expansion \eqref{eq:2.1}  as
\begin{eqnarray*}
  \Omega(t)&\leadsto&
  \begin{picture}(10,12)
    \put (5,0) \DI
    \put (5,0) \UP
    \put (5,12) \DI
  \end{picture} -\frac12
  \begin{picture}(24,36)
    \put (12,0) \DI
    \put (12,0) \UP
    \put (12,12) \DI
    \put (12,12) \LE
    \put (12,12) \RI
    \put (0,24) \DI
    \put (24,24) \DI
    \put (0,24) \UP
    \put (0,36) \DI
  \end{picture}+\frac{1}{12}
  \begin{picture}(36,48)
    \put (12,0) \DI
    \put (12,0) \UP
    \put (12,12) \DI
    \put (12,12) \LE
    \put (12,12) \RI
    \put (0,24) \DI
    \put (24,24) \DI
    \put (0,24) \UP
    \put (0,36) \DI
    \put (24,24) \LE
    \put (24,24) \RI
    \put (12,36) \DI
    \put (36,36) \DI
    \put (12,36) \UP
    \put (12,48) \DI
  \end{picture}+\frac14
  \begin{picture}(36,60)
    \put (24,0) \DI
    \put (24,0) \UP
    \put (24,12) \DI
    \put (24,12) \LE
    \put (24,12) \RI
    \put (12,24) \DI
    \put (36,24) \DI
    \put (12,24) \UP
    \put (12,36) \DI
    \put (12,36) \LE
    \put (12,36) \RI
    \put (0,48) \DI
    \put (24,48) \DI
    \put (0,48) \UP
    \put (0,60) \DI
  \end{picture}-\frac18
  \begin{picture}(48,84)
    \put (36,0) \DI
    \put (36,0) \UP
    \put (36,12) \DI
    \put (36,12) \LE
    \put (36,12) \RI
    \put (24,24) \DI
    \put (48,24) \DI
    \put (24,24) \UP
    \put (24,36) \DI
    \put (24,36) \LE
    \put (24,36) \RI
    \put (12,48) \DI
    \put (36,48) \DI
    \put (12,48) \UP
    \put (12,60) \DI
    \put (12,60) \LE
    \put (12,60) \RI
    \put (0,72) \DI
    \put (24,72) \DI
    \put (0,72) \UP
    \put (0,84) \DI
  \end{picture}-\frac{1}{24}
  \begin{picture}(36,72)
    \put (12,0) \DI
    \put (12,0) \UP
    \put (12,12) \DI
    \put (12,12) \LE
    \put (12,12) \RI
    \put (0,24) \DI
    \put (24,24) \DI
    \put (0,24) \UP
    \put (0,36) \DI
    \put (24,24) \LE
    \put (24,24) \RI
    \put (12,36) \DI
    \put (36,36) \DI
    \put (12,36) \UP
    \put (12,48) \DI
    \put (12,48) \LE
    \put (12,48) \RI
    \put (0,60) \DI
    \put (24,60) \DI
    \put (0,60) \UP
    \put (0,72) \DI
  \end{picture}\\
  &&\mbox{}-\frac{1}{24}
  \begin{picture}(60,60)
    \put (30,0) \DI
    \put (30,0) \UP
    \put (30,12) \DI
    \put (30,12) {\line(-3,2){18}}
    \put (30,12) {\line(3,2){18}}
    \put (12,24) \DI
    \put (48,24) \DI
    \put (12,24) \UP
    \put (12,36) \DI
    \put (12,36) \LE
    \put (12,36) \RI
    \put (0,48) \DI
    \put (24,48) \DI
    \put (0,48) \UP
    \put (0,60) \DI
    \put (48,24) \LE
    \put (48,24) \RI
    \put (36,36) \DI
    \put (60,36) \DI
    \put (36,36) \UP
    \put (36,48) \DI
  \end{picture}-\frac{1}{24}
  \begin{picture}(36,80)
    \put (24,0) \DI
    \put (24,0) \UP
    \put (24,12) \DI
    \put (24,12) \LE
    \put (24,12) \RI
    \put (12,24) \DI
    \put (36,24) \DI
    \put (12,24) \UP
    \put (12,36) \DI
    \put (12,36) \LE
    \put (12,36) \RI
    \put (0,48) \DI
    \put (24,48) \DI
    \put (0,48) \UP
    \put (0,60) \DI
    \put (24,48) \LE
    \put (24,48) \RI
    \put (12,60) \DI
    \put (36,60) \DI
    \put (12,60) \UP
    \put (12,72) \DI
  \end{picture}+\cdots.
\end{eqnarray*}

\subsection{A special property of isomerisation matrices}
Recognising the following special property (confirmed by an easy matrix multiplication)
\begin{equation}
  \label{eq:1.3}
  [A^{[0]},A^{[1]}]=-2A^{[1]}
\end{equation}
usefully simplifies our Magnus expansion.
This simple form of the commutator \R{eq:1.3} is fundamental because the Magnus expansion is constructed as a linear combination of terms that can be obtained from $\AAA(t)=A^{[0]}+A^{[1]}f(t)$ using only integration and commutation.
 It thus resides in the {\em free Lie algebra\/} $\mathcal{F}$ generated by $A^{[0]}$ and $A^{[1]}$.
In light of \R{eq:1.3}, that $\mathcal{F}$ is
\begin{equation}
  \label{eq:1.4}
  \mathcal{F} (A^{[0]},A^{[1]})=\CC{Span}\,\{A^{[0]},A^{[1]}\}.
\end{equation}
In other words, although in general the Magnus expansion of the solution may require many terms, \textit{the Magnus expansion of \R{eq:1.1} for isomerisation is simply a linear combination of the form\footnote{Indeed, more is true.  A Lie algebra $\Gg{g}$ is \textit{solvable} if there exists $M \geq 0 $ such that $\Gg{g}^{[M]}=\{0\}$, where $\Gg{g}^{[0]}=\Gg{g}$ and $\Gg{g}^{[k+1]}=[\Gg{g}^{[k]},\Gg{g}^{[k]}]$. By \R{eq:1.3}, $\dim  \mathcal{F}^{[1]}=1$ so it is a commutative algebra and $ \mathcal{F}^{[2]}=\{0\}$. The algebra is solvable!}  \, $ \Omega(t) = \sigma_{[0]}(t) A^{[0]}+\sigma_{[1]}(t)A^{[1]}$!}

\subsection{A Magnus expansion of isomerisation}

We now specialize the general form of the expansion  \eqref{eq:2.1} to our application of isomerisation  \R{eq:1.1}, for which
\begin{displaymath}
  \bullet\leadsto A^{[0]}+f(t)A^{[1]}.
\end{displaymath}
By following the four step algorithm near \eqref{eq:2.2}, we find the first few terms in the series \eqref{eq:ODE:magnus:first:two:terms} and the corresponding trees are
\begin{eqnarray*}
  \begin{picture}(10,12)
    \put (5,0) \DI
    \put (5,0) \UP
    \put (5,12) \DI
  \end{picture}:&\quad& \int_0^t \AAA(x)\D x=tA^{[0]}+\int_0^t f(x)\D x A^{[1]},\\
  \begin{picture}(24,36)
    \put (12,0) \DI
    \put (12,0) \UP
    \put (12,12) \DI
    \put (12,12) \LE
    \put (12,12) \RI
    \put (0,24) \DI
    \put (24,24) \DI
    \put (0,24) \UP
    \put (0,36) \DI
  \end{picture}:&& \int_0^t \int_0^{x_1} [\AAA(x_2),\AAA(x_1)]\D x_2\D x_1 \\
  &&=\int_0^t \left[x_1f(x_1)-\int_0^{x_1} f(x_2)\D x_2\right]\D x_1 [A^{[0]},A^{[1]}]\\
  &&=2\int_0^t (t-2x)f(x)\D x A^{[1]}
\end{eqnarray*}
and so on.
Note we made use of  \R{eq:1.3} for the commutator to simplify the expressions.
Moreover, a matrix commutes with itself so some  terms are zero, such as
\begin{displaymath}
  \begin{picture}(60,60)
    \put (30,0) \DI
    \put (30,0) \UP
    \put (30,12) \DI
    \put (30,12) {\line(-3,2){18}}
    \put (30,12) {\line(3,2){18}}
    \put (12,24) \DI
    \put (48,24) \DI
    \put (12,24) \UP
    \put (12,36) \DI
    \put (12,36) \LE
    \put (12,36) \RI
    \put (0,48) \DI
    \put (24,48) \DI
    \put (0,48) \UP
    \put (0,60) \DI
    \put (48,24) \LE
    \put (48,24) \RI
    \put (36,36) \DI
    \put (60,36) \DI
    \put (36,36) \UP
    \put (36,48) \DI
  \end{picture}:\qquad \left[ 2\int_0^t (t-2x)f(x)\D x A^{[1]},-2\int_0^t [f(t)-f(x)]\D x A^{[1]}\right]=O.
\end{displaymath}

We claim that for $\tau\in\BB{T}_m$, $m\geq1$, necessarily $G_\tau$ is a scalar multiple of $A^{[1]}$, i.e.
$
  G_\tau(t)=\sigma_\tau(t) A^{[1]}.
$

We already know from  \R{eq:1.3}  and \eqref{eq:1.4} that our Magnus expansion is of the form $\sigma_{[0]}(t) A^{[0]}+\sigma_{[1]}(t)A^{[1]}$.
In view of the first few trees above, our claim immediately implies $\sigma_{[0]} (t)=t$.
Having now found $\sigma_{[0]}$, it remains only to find $\sigma_{[1]}$, so to simplify notation, we drop the subscript from now on and let $\sigma=\sigma_{[1]}$.

The proof of the claim is by induction. 
For $m=1$ there is only one Magnus tree,
\begin{displaymath}
  \tau=
  \begin{picture}(24,24)
    \put (12,0) \DI
    \put (12,0) \LE
    \put (12,0) \RI
    \put (0,12) \DI
    \put (24,12) \DI
    \put (0,12) \UP
    \put (0,24) \DI
  \end{picture}\qquad\Rightarrow\qquad G_\tau(t)=-2\int_0^t [f(t)-f(x)]\D x A^{[1]}.
\end{displaymath}
Therefore $\sigma_\tau(t)=-2\int_0^t [f(t)-f(x)]\D x$.

Consider next $m\geq2$ and \R{eq:2.2}. If $m_1,m_2\geq1$ then, by the induction assumption, both $G_{\tau_1}$ and $G_{\tau_2}$ are scalar multiples of $A^{[1]}$ and we deduce that $G_\tau\equiv O$. 
There are two remaining possibilities: either $m_1=0$, $m_2=m-1$ or $m_1=m-1$, $m_2=0$. 
In the first case
\begin{equation}
  \label{eq:2.3}
  \tau=
  \begin{picture}(24,24)
    \put (12,0) \DI
    \put (12,0) \LE
    \put (12,0) \RI
    \put (0,12) \DI
    \put (0,12) \UP
    \put (0,24) \DI
    \put (20,16) {$\tau_2$}
  \end{picture},
\end{equation}
so
$
  G_\tau(t)=\left[tA^{[0]}+\int_0^t f(x)\D x A^{[1]},\sigma_{\tau_2}(t)A^{[1]}\right]=t\sigma_{\tau_2}(t) [A^{[0]},A^{[1]}]  
$
which is simply $G_\tau(t)=-2t\sigma_{\tau_2}(t) A^{[1]}$, so $\sigma_\tau(t)=-2t\sigma_{\tau_2}(t)$.

Finally, for $m_1=m-1$ and $m_2=0$, we have
\begin{equation}
  \label{eq:2.4}
  \tau=
  \begin{picture}(24,36)
    \put (12,0) \DI
    \put (12,0) \LE
    \put (12,0) \RI
    \put (0,12) \DI
    \put (24,12) \DI
    \put (0,12) \UP
    \put (-4,28) {$\tau_1$}
  \end{picture}
\end{equation}
for which
$
  G_\tau(t)=\left[\int_0^t \sigma_{\tau_1}(x)\D xA^{[1]},A^{[0]}+f(t)A^{[1]}\right]=-\int_0^t \sigma_{\tau_1}(x)\D x[A^{[0]},A^{[1]}].
$
This is simply  $G_\tau(t)=2\int_0^t \sigma_{\tau_1}(x)\D x A^{[1]}$ so
 $\sigma_\tau(t)=2\int_0^t \sigma_{\tau_1}(x)\D x $. 
This completes the proof of 

\begin{theorem}
\label{theorem:magnus:expansion:isomerisation:first:pass}
  The Magnus expansion for isomerisation \R{eq:1.1} is of the form 
  \begin{equation}
  \Omega(t)=tA^{[0]}+\sigma(t) A^{[1]}
  \label{eq:Magnus:expansion:form:first:pass}
  \end{equation}
   for a function $\sigma$ which has been described above in a recursive manner.
\end{theorem}

Next, we will explicitly find the function $\sigma$ of \eqref{eq:Magnus:expansion:form:first:pass}  in the Theorem, and thus find the Magnus expansion of isomerisation.
We do not present all steps in the derivations to come. 
Theorem \eqref{theorem:magnus:expansion:isomerisation:first:pass} and the steps leading to it were deliberately chosen for presentation partly because this quickly gives a good sense of the style of arguments needed in this area, while still being very accessible. 
The steps required in our other proofs follow a similar pattern, albeit more detailed.

\subsection{Constructing the trees}
In general, when we want to find the Magnus trees, we can follow the four-step algorithm near \eqref{eq:2.2}.
That always works.
Often though, particular applications allow simplifications, as we now use our application to illustrate.
The main question to be answered for this example is how to connect the coefficients  $\alpha(\tau)$ to the trees in the situations of \R{eq:2.3} and of \R{eq:2.4}.

The situation for \R{eq:2.4} is trivial: since $s=1$, we have
\begin{displaymath}
  \alpha(\tau)=\frac{\CC{B}_1}{1!}\alpha(\tau_1)=-\frac12 \alpha(\tau_1).
\end{displaymath}

It is more complicated in the situation of \R{eq:2.3}. 
There we have
\begin{displaymath}
  \tau_2=
  \begin{picture}(60,72)
    \put (12,0) \DI
    \put (12,0) \LE
    \put (12,0) \RI
    \put (0,12) \DI
    \put (0,12) \UP
    \put (-4,28) {$\eta_1$}
    \put (24,12) \DI
    \put (24,12) \LE
    \put (24,12) \RI
    \put (12,24) \DI
    \put (12,24) \UP
    \put (8,40) {$\eta_2$}
    \put (48,36) \DI
    \put (48,36) \LE
    \put (48,36) \RI
    \put (36,48) \DI
    \put (60,48) \DI
    \put (36,48) \UP
    \put (32,64) {$\eta_s$}
    \multiput (38,26)(2,2){4}{\circle*{1}}
  \end{picture}\qquad\Rightarrow\qquad
  \tau=
  \begin{picture}(72,76)
    \put (12,0) \DI
    \put (12,0) \LE
    \put (12,0) \RI
    \put (0,12) \DI
    \put (0,12) \UP
    \put (0,24) \DI
    \put (24,12) \DI
    \put (24,12) \LE
    \put (24,12) \RI
    \put (12,24) \DI
    \put (12,24) \UP
    \put (8,40) {$\eta_1$}
    \put (36,24) \DI
    \put (36,24) \LE
    \put (36,24) \RI
    \put (24,36) \DI
    \put (24,36) \UP
    \put (20,52) {$\eta_2$}
    \put (60,48) \DI
    \put (60,48) \LE
    \put (60,48) \RI
    \put (48,60) \DI
    \put (72,60) \DI
    \put (48,60) \UP
    \put (44,76) {$\eta_s$}
    \multiput (50,38)(2,2){4}{\circle*{1}}
  \end{picture}
\end{displaymath}
Therefore
\begin{displaymath}
  \alpha(\tau_2)=\frac{\CC{B}_s}{s!} \prod_{j=1}^s \alpha(\eta_j),\qquad \alpha(\tau)=\frac{\CC{B}_{s+1}}{(s+1)!} \prod_{j=1}^s \alpha(\eta_j).
\end{displaymath}
Hence, to summarize  
\begin{eqnarray*}
  s=1:& & \quad \alpha(\tau_2)=-\frac12 \alpha(\eta_1),\quad \alpha(\tau)=\frac{1}{12} \alpha(\eta_1)=-\frac16 \alpha(\tau_2);\\
  s\mbox{\ even}:& & \quad \CC{B}_{s+1}=0\quad\Rightarrow\quad \alpha(\tau)=0;\\
  s\geq3\mbox{\ odd}: & &  \quad \CC{B}_s=0\quad\Rightarrow\quad \alpha(\tau_2)=0.
\end{eqnarray*}

This is a moment to comment on the mechanisms giving rise to some of our simplifications. 
Not all Magnus trees feature --- with nonzero coefficients --- in the expansion \R{eq:2.1}.
There are two mechanisms that explain this:
(i) The coefficient $\alpha(\tau)$ is zero; or
(ii) $\sigma_\tau\equiv0$, because a matrix commutes with itself and $\tau$ originates in trees $\tau_1$ and $\tau_2$ such that   $G_{\tau_k}(t)=\sigma_{\tau_k}(t)A^{[1]}$, for $k=1,2$.
There is an important difference between these two situations. 
For the first mechanism, while we do not include the tree $\tau$ in \R{eq:2.1}, we must retain it for further recursions. 
In the second mechanism, though, if a tree is zero then all its `children' are zero too.

The long-and-short is that in every $\BB{T}_m$, $m\geq1$ we have $2^{m-1}$ trees (some with a zero coefficient).
What we really have is a binary `super-tree'
\vspace{-0.2cm}
\begin{displaymath}
  \begin{picture}(300,200)
    \thicklines
    \put (150,190) {$\tau_\star$}
    \put (148,187) {\vector(-2,-1){50}}
    \put (158,187) {\vector(2,-1){50}}
    \put (93,153) {$\tau_0$}
    \put (205,153) {$\tau_1$}
    \put (92,150) {\vector(-2,-3){20}}
    \put (102,150) {\vector(2,-3){20}}
    \put (203,150) {\vector(-2,-3){20}}
    \put (213,150) {\vector(2,-3){20}}
    \put (67,112) {$\tau_{00}$}
    \put (118,112) {$\tau_{10}$}
    \put (176,112) {$\tau_{01}$}
    \put (227,112) {$\tau_{11}$}
    \put (65,108) {\vector(-1,-2){12}}
    \put (75,108) {\vector(1,-2){12}}
    \put (116,108) {\vector(-1,-2){12}}
    \put (126,108) {\vector(1,-2){12}}
    \put (174,108) {\vector(-1,-2){12}}
    \put (184,108) {\vector(1,-2){12}}
    \put (225,108) {\vector(-1,-2){12}}
    \put (235,108) {\vector(1,-2){12}}
    \put (44,76) {$\tau_{000}$}
    \put (77,76) {$\tau_{100}$}
    \put (97,76) {$\tau_{110}$}
    \put (130,76) {$\tau_{010}$}
    \put (153,76) {$\tau_{001}$}
    \put (186,76) {$\tau_{101}$}
    \put (206,76) {$\tau_{011}$}
    \put (239,76) {$\tau_{111}$}
    \put (115,60) {\ldots{}and so on.}
  \end{picture}
\end{displaymath}

  \vspace{-2.2cm} \noindent The rule is: Each move `left' (i.e.\ in the 0 direction -- the subscripts are binary strings) corresponds to `scenario' \R{eq:2.3}; 
Each move `right' corresponds to `scenario' \R{eq:2.4}.
Now that we have simplified our system for dealing with the trees,  we are ready to proceed to find $\sigma$.

\vspace{-0.1cm}
\subsection{An explicit formula for $\sigma$}
As we have seen, except for $\BB{T}_0$, every $\tau\in\BB{T}_m$ leads to an expression of the form $\sigma_\tau(t) A^{[1]}$. 
For example, setting $\tilde{f}(x)=xf'(x)$,
\begin{eqnarray*}
  \BB{T}_1 : && \;\;\;\;\;\; \tau_\star=
  \begin{picture}(24,24)
    \put (12,0) \DI
    \put (12,0) \LE
    \put (12,0) \RI
    \put (0,12) \DI
    \put (24,12) \DI
    \put (0,12) \UP
    \put (0,24) \DI
  \end{picture}\quad\Rightarrow\quad  \sigma_{\tau_\star}=-2\int_0^t \tilde{f}(x)\D x,\quad \alpha(\tau_\star)=-\frac12.  
 \end{eqnarray*}
By continuing to find these trees, we see a pattern emerge:
For any $\tau\in\BB{T}_m$, $m\geq1$, our $\sigma_\tau(t)$ is of the form
$
  \sigma_\tau(t)=\int_0^t K_\tau(t,x)\tilde{f}(x)\D x
$
for some {\em kernel\/} $K_\tau$.
To find the kernels, it is convenient for $\tau\in\BB{T}_m$, $m\geq2$, to work with
\begin{equation}
  \label{eq:2.6}
  \tau=
  \begin{picture}(60,72)
    \put (12,0) \DI
    \put (12,0) \LE
    \put (12,0) \RI
    \put (0,12) \DI
    \put (24,12) \DI
    \put (0,12) \UP
    \put (0,24) \DI
    \put (24,12) \LE
    \put (12,24) \DI
    \put (12,24) \UP
    \put (12,36) \DI
    \multiput (27,15)(2,2){4} {\circle*{1}}
    \put (36,24) \DI
    \put (36,24) \LE
    \put (36,24) {\line(1,1){24}}
    \put (24,36) \DI
    \put (48,36) \DI
    \put (24,36) \UP
    \put (24,48) \DI
    \put (48,36) \LE
    \put (36,48) \DI
    \put (60,48) \DI
    \put (36,48) \UP
    \put (34,64) {$\eta$}
    \put (-15,26) {\rotatebox{45}{$\overbrace{\hspace*{45pt}}^{r\CC{\ times}}$}}
  \end{picture}
\end{equation}
Let $r\in\{0,1,\ldots,m-2\}$ and $\eta\in\BB{T}_{m-r}$. 
Straightforward computation shows that
\begin{eqnarray*}
  \eta&\leadsto& K_\eta(t,x), \qquad 
  \begin{picture}(10,24)
    \put (5,0) \DI
    \put (5,0) \UP
    \put (3,16) {$\eta$}
  \end{picture} \leadsto  \int_x^t K_\eta(y,x)\D y,
\qquad
  \begin{picture}(24,30)
    \put (12,0) \DI
    \put (12,0) \LE
    \put (12,0) \RI
    \put (0,12) \DI
    \put (24,12) \DI
    \put (0,12) \UP
    \put (-3,28) {$\eta$}
  \end{picture} \leadsto 2\int_x^t K_\eta(y,x)\D y .
\end{eqnarray*}
%
This pattern motivates arguments by induction, for \R{eq:2.6}, that lead to
\begin{equation}
  \label{eq:2.7}
  K_\tau(t,x)=2(-2t)^r \int_x^t K_\eta(y,x)\D y,\qquad \alpha(\tau)=\frac{\CC{B}_{r+1}}{(r+1)!}\alpha(\eta).
\end{equation}
We left out one exceptional case, namely $\tau=\tau_{\Mm{0}}$. 
In that case the representation \R{eq:2.6} is still true but $\eta\in\BB{T}_0$, so is not associated with a kernel. 
However, easy computation confirms that
$
  K_{\tau_{\Mm{0}}}(t,x)=-2(-2t)^{m-1},\; \alpha(\tau_{\Mm{0}})=\frac{\CC{B}_m}{m!}.
$

Now that we have the kernels, we sum them. 
Let
\[
  \Theta_m(t,x)=\sum_{\tau\in\bb{T}_m} \alpha(\tau) K_\tau(t,x),
\]
for $m\in\BB{N}$.
For example, $  \Theta_1(t,x) \equiv 1$ and $  \Theta_2(t,x) =-\frac23t+x. $ 
Next, let
$
  \Theta(t,x)=\sum_{m=1}^\infty \Theta_m(t,x).
$
After some recursion we are led to the Volterra-type equation
\begin{equation}
  \label{eq:2.10}
  \frac{t(1-\ee^{-2t})}{1-2t-\ee^{-2t}} \Theta(t,x)=\int_x^t \Theta(y,x)\D y-1,
\end{equation}
with solution  
\begin{equation}
  \label{eq:2.11}
  \Theta(t,x)=-\exp\!\left(-4\int_x^t \frac{1-y-(1+y)\ee^{-2y}}{(1-\ee^{-2y})(1-2y-\ee^{-2y})}\D y\right) \frac{1-2x-\ee^{-2x}}{x(1-\ee^{-2x})}\D\xi.
\end{equation}
Finally, we  integrate the contribution of the individual $\sigma_\tau$s, scaled by $\alpha(\tau)$, from each tree, for all Magnus trees:
$  
\sigma(t) = \int_0^t \sum_{m=0}^\infty \sum_{\tau\in\bb{T}_m} \alpha(\tau)\sigma_\tau(\xi)\D\xi = \int_0^t f(x)\D x+\int_0^t \sum_{m=1}^\infty \sum_{\tau\in\bb{T}_m} \alpha(\tau) \int_0^\xi K_\tau(\xi,x)\tilde{f}(x)\D x\D\xi.
$
Swapping integration and summation, we have
$
\sigma(t) = \int_0^t f(x)\D x+\int_0^t xf'(x) \int_x^t \Theta(\xi,x)\D\xi\D x.
$
Substituting \R{eq:2.10},
 we attain our desired goal
  \[
  \sigma(t)=\int_0^t f(x)\D x+\int_0^t xf'(x) \left[ \frac{t(1-\ee^{-2t})}{1-2t-\ee^{-2t}} \Theta(t,x)+1\right]\!\D x ,
  \]
   or
\begin{eqnarray}
  \label{eq:2.12}
   \sigma(t) &=&tf(t)+\frac{t(1-\ee^{-2t})}{1-2t-\ee^{-2t}} \int_0^t xf'(x)\Theta(t,x)\D x.
\end{eqnarray}
Here we used integration by parts,
$
  \int_0^t xf'(x)\D x=tf(t)-\int_0^t f(x)\D x.
$
With \eqref{eq:2.11}, everything is now explicit.
Combining $\sigma$  in  \eqref{eq:2.12} with Theorem \eqref{theorem:magnus:expansion:isomerisation:first:pass}, we have now found the (complete!) Magnus expansion of isomerisation.

Note that \R{eq:2.12} is bounded for all $t\geq0$, because $t(1-\ee^{-2t})/(1-2t-\ee^{-2t})$ is bounded\footnote{Actually, it is analytic.} for all $t\in\BB{R}$.
As a consequence, \textit{the Magnus series \eqref{eq:Magnus:expansion:form:first:pass} for isomerisation converges for every $t\geq0$.}
That is a significant finding for isomerisation, because in general the Magnus series is only convergent for small times.

There is further significance. 
Our own exposition of the Magnus expansion here  also explains the intriguing numerical evidence appearing in earlier work that time-steps larger than the Moan--Niesen sufficient condition for convergence of the Magnus expansion can be taken while still maintaining good accuracy with Magnus-based numerical methods \cite[Figure 1]{KormannMacNamara2016}.
That good experience of taking larger time steps with Magnus-based methods has  previously been reported in numerous numerical studies in the context of the Schr\"odinger equation, and was eventually carefully explained by Hochbruck and Lubich \cite{HochbruckLubich03}.
We are also seeing it here in a novel context of master equations, although our explanation via the Magnus expansion shows that same good experience in this novel context is for completely different reasons.

\subsection{A role for automorphisms}
Theorem \eqref{theorem:magnus:expansion:isomerisation:first:pass} and  \eqref{eq:2.12} tell us the answer to the question of  the matrix $\Omega(t)$ in the Magnus expansion.
Ultimately, we want the solution \eqref{eq:ODE:magnus:form}. 
For that, we need the exponential, $\exp( \MM{\Omega}(t) )$.
This is an opportunity to show how automorphisms can simplify exponentials arising in master equations.

Let $P$ be the $(N+1)\times(N+1)$ {\em persymmetric identity}: $P_{i,j}=1$ if $j=N-i$, and is zero otherwise. 
Note $P\in\CC{O}(N+1)\cap\CC{Sym}(N+1)$ so $P$ is an {\em orthogonal involution:\/} $P^{-1}=P^\top=P$ and $P^2=I$.
Matrix multiplication confirms the useful properties
\begin{equation}
  \label{eq:3.1}
  PA^{[0]}P=A^{[0]},\qquad PA^{[1]}P=-A^{[1]}.
\end{equation}

Being an orthogonal involution, $P$ defines an inner automorphism on $\GG{gl}(N+1)$, namely
$
  \iota(B)=PBP$
  for 
$
B\in\GG{gl}(N+1).
$
Following \cite{munthekass01gpd}, we let
$
  \GG{k}=\{B\in\GG{gl}(N+1)\,:\, \iota(B)=B\} 
$
and
$
\GG{p}=\{B\in\GG{gl}(N+1)\,:\, \iota(B)=-B\}
$
be the {\em fixed points\/} and {\em anti-fix points\/} of the automorphism $\iota$. 
Here is a  list of the three main features of our general strategy. 
First,  in the {\em Generalised Cartan Decomposition\/},  $\GG{gl}(N+1)=\GG{k}\oplus\GG{p}$. 
That is, given $B\in\GG{gl}(N+1)$, we split it into
$
  \frac12[B+\iota(B)]\in\GG{k}
$
and 
$
 \frac12[B-\iota(B)]\in\GG{p}.
$
Second, here $\GG{k}$ is a subalgebra of $\GG{gl}(N+1)$, while $\GG{p}$ is a {\em Lie triple system:\/}
$
  [\GG{k},\GG{k}],[\GG{p},\GG{p}]\subseteq\GG{k}
$
and
$
[\GG{k},\GG{p}],[\GG{p},\GG{k}]\in\GG{p}.
$
Third, letting $B=k+p$ where $k\in\GG{k}$ and $p\in\GG{p}$, we have 
 (and we will apply this form to our example momentarily)
\[
\ee^{tB}=\ee^X\ee^Y,
\]
 where $X\in\GG{k}$, $Y\in\GG{p}$ have the Taylor expansion
\begin{eqnarray}
  \label{eq:3.2}
  X&=&tp-\frac12t^2 [p,k]-\frac16 t^3 [k,[p,k]] +t^4 \left(\frac{1}{24}[p.[p,[p.k]]]-\frac{1}{24}[k,[k,[p,k]]]\right)\hspace*{20pt}\\
  \nonumber
  &&\mbox{}+t^5\left(\frac{7}{360}[k,[p,[p,[p,k]]]]-\frac{1}{120}[k,[k,[k,[p,k]]]]-\frac{1}{180} [[p,k],[p,[p,k]]]\right)\\
  \nonumber
  &&\mbox{}+t^6\left(-\frac{1}{240}[p,[p,[p,[p,[p,k]]]]]+\frac{1}{180} [k,[k,[p,[p,[p,k]]]]]\right.\\
  \nonumber
  &&\hspace*{20pt}\mbox{}-\frac{1}{720} [k,[k,[k,[k,[p,k]]]]]+\frac{1}{720}[[p,k],[k,[p,[p,k]]]]\\
  \nonumber
  &&\hspace*{20pt}\left.\mbox{}+\frac{1}{180} [[p,[p,k]],[k,[p,k]]]\right)+\mathcal{O}(t^7),\\
  \label{eq:3.3}
  Y&=&tk-\frac{1}{12}t^3[p,[p,k]]+ t^5\left(\frac{1}{120}[p,[p,[p,[p,k]]]] +\frac{1}{720} [k,[k,[p,[p,k]]]]\right.\\
  \nonumber
  &&\hspace*{20pt}\left.\mbox{}-\frac{1}{240} [[p,k],[k,[p,k]]] \right)+\mathcal{O}(t^7).
\end{eqnarray}

Now, let $k=A^{[0]}$ and $p=A^{[1]}$ so by \R{eq:1.3},  $[p,k]=2p$. 
Look again at \R{eq:3.2} and \R{eq:3.3}. 
Each term necessarily contains the commutator $[p,k]$. 
Suppose that, except for this commutator, the term contains at least one additional $p$. 
Then, necessarily, it is zero. 
The reason is there must be a sub-term of the form
$
  [p,[k,[k,[\ldots,[k,[p,k]]\cdots]]]].
$
Beginning from the inner bracket, we replace $[p,k]$ by $2p$, so $[k,[p,k]]=-4p$, and so on, until we reach the outermost commutator: up to a power of 2, it will be $[p,p]=0$, proving our assertion. 
We deduce that the only terms surviving in \R{eq:3.2}, except for the first, are of the form (where in this line we are also introducing an adjoint operator notation $\CC{ad}_k^{r+1}$, to simplify expressions with nested commutators)
\begin{displaymath}
  [\overbrace{k,[k,\cdots,k}^{r\geq0\CC{\ times}},[p,k]]]=-\CC{ad}_k^{r+1}p =(-1)^{r} 2^{r+1} p 
\end{displaymath}
so
\begin{equation}
  \label{eq:3.4}
  X=-\sum_{r=1}^\infty \frac{t^r}{r!} \CC{ad}_k^{r-1} p=\frac{1-\ee^{-2t}}{2} p.
\end{equation}

Insofar as $Y$ is concerned, things are even simpler.
 While $p$ features an odd number of times in $X$ (because $X\in\GG{k}$), $Y\in\GG{p}$ implies that $p$ features there an even number of times. 
Except for the leading term, it features at least twice, and each such term must vanish, so
\begin{equation}
  \label{eq:3.5}
  Y=tk.
\end{equation}

Of course, what we  really need to compute is $\exp( \MM{\Omega}(t) ) = \exp(tA^{[0]}+\sigma(t) A^{[1]}) = \ee^{tB}=\ee^X\ee^Y$. 
For that, we  keep \R{eq:3.5} intact (hence $Y=tA^{[0]}$), but $t$ in \R{eq:3.4} need be replaced by $\sigma(t)/t$ (which is not problematic since $\sigma(0)=0$), i.e.
\begin{displaymath}
  X=\frac12 \left[1-\exp\!\left(-\frac{2\sigma(t)}{t}\right)\right] A^{[1]}.
\end{displaymath}
Thus automorphisms have simplified the required $\exp(tA^{[0]}+\sigma(t) A^{[1]}) $ to computing exponentials of $A^{[0]}$ and of $A^{[1]}$ separately.
Those come from the spectral decomposition, which we set about finding next.

\section{Spectra and pseudospectra of isomerisation matrices}
\label{sec:spectrum}

\subsection{Spectral decomposition of $A^{[0]}$}
We wish to determine the eigenvalues and eigenvectors of $A^{[0]}$. 
They are essentially given by  \cite[Theorem 2.1]{EdelmanRoadfromKactoKac1994}. 
Here we provide an alternative proof and a formula for the eigenvectors.
\begin{theorem}
\label{theorem:eigenvalues:A0}
  The spectrum of $A^{[0]}$ is 
  \[
  \{-2r\,:\,r=0,1,\ldots,N\}.
  \] 
   Moreover, an (unnormalised) eigenvector corresponding to the eigenvalue $-2r$, for $r=0,\ldots,N$, is
  \begin{eqnarray}
    \label{eq:3.6}
    v_m&=&(-1)^m {r\choose m} \hyper{2}{1}{-N+r,-m}{r-m+1}{-1},\qquad m=0,\ldots,r,\\
    \label{eq:3.7}
    v_m&=&(-1)^r {{N-r}\choose{m-r}} \hyper{2}{1}{-N+m,-r}{m-r+1}{-1},\qquad m=r,\ldots,N.
  \end{eqnarray}
  where ${}_kF_\ell$ is the generalized hypergeometric function.
\end{theorem}

\begin{proof}
By definition, $\lambda$ is an eigenvalue of $A^{[0]}$ and $\MM{v} \neq \MM{0}$  a corresponding eigenvector if and only if
  \begin{equation}
    \label{eq:3.8}
    (N+1-m)v_{m-1}-(N+\lambda)v_m+(m+1)v_{m+1}=0,\qquad m=0,\ldots,N,
  \end{equation}
  with the boundary conditions $v_{-1}=v_{N+1}=0$. 
 One way to arrive at the theorem is to let
  \begin{displaymath}
    \mathcal{V}(t) := \sum_{m=0}^N v_m t^m
  \end{displaymath}
  and establish $\mathcal{V}=(1+t)^{N+\lambda/2}(1-t)^{-\lambda/2}$ using \R{eq:3.8}.
   Then impose conditions on $\lambda$ to ensure $\mathcal{V}$ is a polynomial of degree $N$. 
  The exact details of the eigenvectors $\MM{v}$ can come by expanding $(1+t)^{N+\lambda/2}(1-t)^{-\lambda/2}$.
\end{proof}

Incidentally, \R{eq:3.6}--\R{eq:3.7} reveal symmetry. 
Denoting the eigenvector corresponding to the eigenvalue $-2r$ by $\MM{v}^{[r]}$, we have: 
$
  v_{N-m}^{[r]}=(-1)^{m-r}v_m^{[N-r]}, \; m=0,\ldots,N.
$

What else can we say about the eigenvector matrix $V=[\MM{v}^{0]},\MM{v}^{[1]},\ldots,\MM{v}^{[N]}]$? 
Computer experiments seem to demonstrate the remarkable result $V^2=2^N I$, hence
\begin{equation}
  \label{eq:3.9}
  V^{-1}=2^{-N}V
\end{equation}
and this is true: for brevity we omit the proof.
More importantly, having the spectral decomposition and having $V^{-1}$, we now have the exponential, exactly:
\begin{displaymath}
  \ee^{tA^{[0]}}=\frac{1}{2^N} V\Lambda(t)V, \qquad \mbox{where}\qquad \Lambda(t)=\CC{diag}\, \left(1,\ee^{-2t},\ee^{-4t},\cdots,\ee^{-2Nt} \right).
\end{displaymath}
It is tempting to compute matrix exponentials via diagonalization.
In general, this is not necessarily a good numerical choice, even in situations where the spectral decomposition is cheaply available.
An issue is that the condition number of the eigenvector matrix can be very large, as happens here\footnote{In hindsight, such poor conditioning of the eigenvector matrix was to be expected because $A^{[0]}$ exhibits a humongous pseudospectrum.
The best case scenario is when eigenvectors form an orthogonal basis (consistent with our intuition from numerical linear algebra that orthogonal matrices have the ideal condition number of $1$), as happens in the real symmetric case.
Pseudospectra measures the departure of a \textit{nonnormal matrix} from that good orthogonal case.
Our example has eigenvectors in Theorem~\ref{theorem:eigenvalues:A0} that are far from orthogonal.} ---  $\kappa(V)$ grows quickly with $N$.
Also, expressions such as $\ee^{-2Nt}$ are at risk of underflow error.

\subsection{A Jordan form of $A^{[1]}$}
Unlike $A^{[0]}$, the matrix $A^{[1]}$ is not diagonalizable.
It can still be usefully factorized in
\begin{theorem}
The Jordan form of $A^{[1]}$ is 
\begin{equation}
  \label{eq:3.12}
  A^{[1]}=WEW^{-1}, 
\end{equation}
where $E$ is the standard {\em shift matrix\/}, with $E_{i,j}= 1$ if $j=i+1$ and is zero otherwise,  while $W$ is a lower-triangular matrix,
\begin{displaymath}
  W_{m,n}=
  \begin{case}
    0, & m\leq n-1,\\[4pt]
    \displaystyle \frac{(-1)^{m-n}}{n!} {{N-n}\choose{m-n}}, & m\geq n,
  \end{case}\qquad m,n=0,\ldots,N.
\end{displaymath}
An immediate consequence of this Jordan form  \eqref{eq:3.12}  is that $A^{[1]}$ is \textit{nilpotent}.
\end{theorem}

\begin{proof}
The Jordan form \R{eq:3.12} is equivalent to $A^{[1]}W=WE$ and the latter is easier to check.
 The matrix $WE$ is easy to find because $E$ is the shift matrix: each column of $W$ is shifted rightwards, the $N$th column disappears, and the zeroth column is replaced by zeros, so
\begin{displaymath}
  (WE)_{m,n}=
  \begin{case}
    0, & n=0,\\
    W_{m,n-1}, & n=1,\ldots,N.
  \end{case}
\end{displaymath}
We proceed to evaluate $A^{[1]}W$ and demonstrate that it is the same.

For every $m,n=0,\ldots,N$ (and with $A^{[1]}_{0,-1}=A^{[1]}_{N,N+1}=0$) we have
\begin{displaymath}
  (A^{[1]}W)_{m,n}=A^{[1]}_{m,m-1}W_{m-1,n}+A^{[1]}_{m,m}W_{m,n}+A^{[1]}_{m,m+1}W_{m+1,n}.
\end{displaymath}
For $n\geq m+2$ this obviously vanishes. 
For $n=m+1$,
$
  A_{m,m+1}^{[1]}W_{m+1,m+1}=\frac{1}{m!}=W_{m,m}
$
is all that survives,  and for $n=m$
\begin{displaymath}
  A_{m,m}^{[1]}W_{m,m}+A^{[1]}_{m,m+1}W_{m+1,m}=-m\frac{N+1-m}{m!}=
  \begin{case}
    0, & m=0,\\
    W_{m,m-1}, & m\geq1.
  \end{case}
\end{displaymath}
Finally, for $n\leq m-1$ all three terms are nonzero and their sum is
\begin{eqnarray*}
  &&(-N+m-1)\frac{(-1)^{m-1-n}}{n!} {{N-n}\choose{m-1-n}} +(N-2m)\frac{(-1)^{m-n}}{n!} {{N-n}\choose{m-n}} \\
  &&\mbox{}+(m+1)\frac{(m+1-n)}{n!} {{N-n}\choose{m+1-n}}\\
  &=&\frac{(-1)^{m-n} n(N-n+1)!}{n!(m-n+1)!(N-m)!}=
  \begin{case}
    0, & n=0,\\
    W_{m,n-1}, & n\geq1
  \end{case}
\end{eqnarray*}
and we are done.
\end{proof}

Next, we set about applying our newly found Jordan form to find the matrix exponential.
Let $ C = \CC{diag}\, \left(0!,1!,2!,\cdots,N! \right)$ be a diagonal matrix 
and
\begin{displaymath}
  Z_{m,n} =
  \begin{case}
    0, & m\leq n-1,\\[4pt]
    \displaystyle (-1)^{m-n} {{N-n}\choose{m-n}}, & m\geq n,
  \end{case}\qquad m,n=0,\ldots,N.
\end{displaymath}
As is trivial to verify, $W=ZC^{-1}$, so
$
  A^{[1]} = ZC^{-1}ECZ^{-1}.
$
Equally trivial to verify is that $Z^{-1}$ is given by 
\begin{displaymath}
  Z^{-1} = \tilde{Z}_{m,n} :=
  \begin{case}
    0, & m\leq n-1,\\[4pt]
    \displaystyle  {{N-n}\choose{m-n}}, & m\geq n,
  \end{case}\qquad m,n=0,\ldots,N.
\end{displaymath}
Consequently,
$
  A^{[1]}=ZC^{-1}EC\tilde{Z}.  
$
We have proved
\begin{theorem}
 The matrix exponential is, in an explicit form,
\begin{equation}
  \label{eq:3.13}
 \ee^{tA^{[1]}}=ZC^{-1}\ee^{tE}C\tilde{Z}.
\end{equation}
\end{theorem}

\begin{figure}
  \centering
  \includegraphics[scale=0.9]{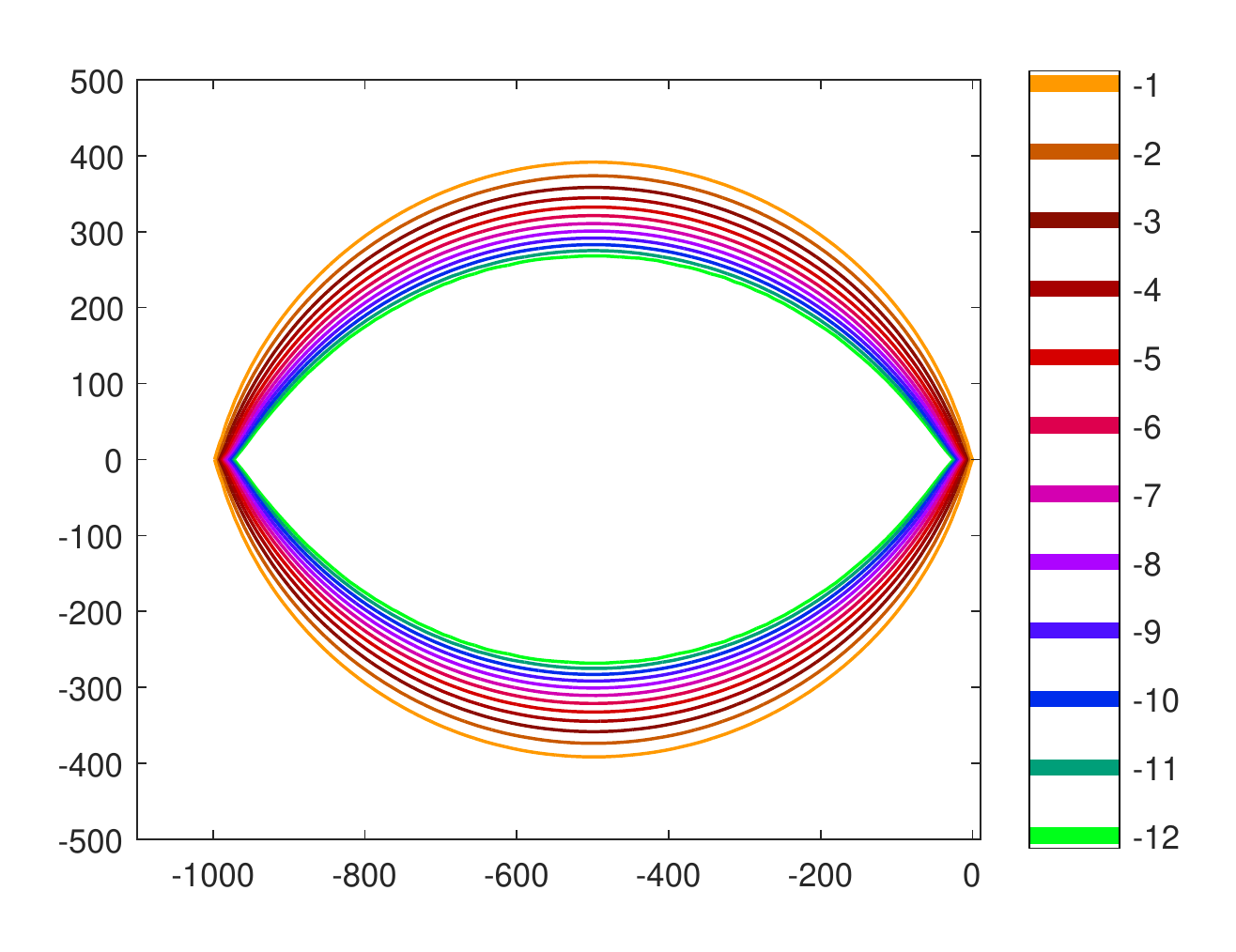}
  \caption{An `almond eye:' Pseudospectrum \cite{TreEmb05} of a $500 \times 500$ example of the $A^{[0]}$ matrix, defined in \eqref{eq:A0:A1:definitions},  as computed by Eigtool \cite{EigTool2002}. 
  Contours of the minimum singular value, $s_{\textrm{min}}(zI-A)$, are displayed on a log scale.  }
  \label{fig:A0}
\end{figure}

\subsubsection{Evaluating the exponential via \eqref{eq:3.13}}
Let $\MM{u}\in\BB{R}^{N+1}$ (again, indexed from zero).
We wish to compute $\MM{y}=\tilde{Z}\MM{u}$. 
A na\"ive approach would require $\mathcal{O}(N^2)$ flops but herewith an algorithm that accomplishes this in just $\mathcal{O}(N^2)$ {\em additions,\/} without requiring multiplications!

For reasons that become clear, it is useful to indicate $N$ explicitly in the notation, i.e.\ $\MM{y}^{[N]}=\tilde{Z}^{[N]}\MM{u}^{[N]}$. 
Start by observing that
\begin{displaymath}
  y_m^{[N]}=\sum_{n=0}^m {{N-n}\choose{m-n}} u_n,\qquad m=0,\ldots,N
\end{displaymath}
(no need to place superscripts on $u_n$). 
Therefore, for $m=0,\ldots,N-1$,
\begin{eqnarray*}
  y_m^{[N]}+y_{m+1}^{[N]}&=&\sum_{n=0}^m{{N-n}\choose{m-n}}u_n+\sum_{n=0}^{m+1} {{N-n}\choose{m+1-n}} u_n=\sum_{n=0}^{m+1}{{N+1-n}\choose{m+1-n}} u_n\\
  &=&y_{m+1}^{[N+1]}.
\end{eqnarray*}
Rewrite this as
\begin{equation}
  \label{eq:3.14}
  y_m^{[N]}=y_{m-1}^{[N-1]}+y_m^{[N-1]},\qquad m=0,\ldots,N-1
\end{equation}
(in the case $m=0$ of course $y_0^{[N]}=y_0^{[N-1]}=u_0$, so the above is consistent with $y_{-1}^{[N]}=0$.) 
Now proceed from $y_0^{[0]}=u_0$ and then, for $M=1,2,\ldots,N$, add
\begin{eqnarray*}
  y_m^{[M]}&=&y_{m-1}^{[M-1]}+y_m^{[M-1]},\qquad m=0,\ldots,M-1,\\
  y_M^{[M]}&=&\sum_{n=0}^M u_n=y_{M-1}^{[M-1]}+u_M.
\end{eqnarray*}
and we are done.

 Of course, similar reasoning applies also to a product $\MM{y}=Z\MM{u}$. 
 The only difference vis-\'a-vis \R{eq:3.14} is that now
 $
  y_m^{[N]}=y_m^{[N-1]}-y_{m-1}^{[N-1]},\; m=0,\ldots,N-1,
$
therefore the recursion steps are
\begin{eqnarray*}
  y_m^{[M]}&=&y_m^{[M-1]}-y_{m-1}^{[M-1]},\qquad m=0,\ldots,N-1,\\
  y_M^{[M]}&=&\sum_{m=0}^M (-1)^{M-n}u_n=-y_{M-1}^{[M-1]}+u_N.
\end{eqnarray*}

Having dealt with the $\tilde{Z}\MM{u}$ and the $Z\MM{u}$ components, we are left only with the $C^{-1}\ee^{tE}C$ portion of \eqref{eq:3.13}.
 We address that now.
 It is  trivial that
\begin{displaymath}
  (\ee^{tE})_{m,n}=
  \begin{case}
    \displaystyle \frac{t^{n-m}}{(n-m)!}, & m=0,\ldots,n,\\[8pt]
    0, & m=n+1,\ldots,N.
  \end{case}
\end{displaymath}
Therefore (cf.\ \R{eq:3.13})
\begin{displaymath}
  (C^{-1}\ee^{tE}C)_{m,n}=
  \begin{case}
    \displaystyle {n\choose m} t^{n-m} & m=0,\ldots,n,\\[12pt]
    0, & m=n+1,\ldots,N.
  \end{case}
\end{displaymath}

Let us pause to reflect on the exact exponentials that we have just found.
We expect the solution to our model of isomerisation to be a binomial distribution \cite{Jahnke2007}.
In general, that means we expect a linear combination of the columns of the solution matrix $\exp (\Omega (t) )$   to be a binomial distribution, when the weights in that linear combination likewise come from a binomial distribution.
Perhaps the simplest example is that the first column of the solution of \R{eq:1.1}  must be a binomial distribution.

As an example, set $\MM{e}_0= (1,0, \ldots, 0)^{\top}$ and compute the leading column, $\ee^{q A^{[1]}}\MM{e}_0=Z(C^{-1}\ee^{q E}C)\tilde{Z}\MM{e}_0$.
Note that  $(\tilde{Z}\MM{e}_0)_m=\tilde{Z}_{m,0}={N\choose m}$.
So 
\[
[(C^{-1}\ee^{q E}C)\tilde{Z}\MM{e}_0]_m = {N\choose m} \sum_{n=0}^{N-m}{{N-m}\choose n}t^n={N\choose m} (1+q)^{N-m}
\]
 and after some simplifications,
\[
(\ee^{q A^{[1]}}\MM{e}_0)_m =[Z(C^{-1}\ee^{q E}C)\tilde{Z}\MM{e}_0] =(-1)^m  {N\choose m}q^m (1+q)^{N-m}.
\]
We are seeing on the right that the binomial distribution survives the first term in $X(t)=\ee^{tA^{[0]}}\ee^{q A^{[1]}}\MM{e}_0$, where $q=\sigma(t)/t$.
Thus, the explicit forms of our exponentials that we have derived allow us to confirm the  `binomial stays binomial' theorem  \cite{Jahnke2007}.

\subsection{Pseudospectra}
\label{sec:pseudospectra}
Having established exact analytic formul\ae{} for  spectral decomposition, we are now in a good position to compare  exact spectra to numerical estimates of the \textit{pseudospectra}  \cite{TreEmb05}.
Two striking contrasts between the numerically computed eigenvalues and the exact eigenvalues are worth pointing out.

First, we proved the matrix $A^{[1]}$ is nilpotent: \textit{exact eigenvalues are precisely zero}.
Nonetheless, $A^{[1]}$  has an enormous pseudospectrum, and standard numerical methods lead to wrongly computed  non-zero eigenvalues of a large magnitude.

Second, we found the eigenvalues of  $A^{[0]}$   in Theorem~\ref{theorem:eigenvalues:A0}, and they are \textit{purely real}. 
(Indeed, the same ideas described by Trefethen and Embree \cite{TreEmb05}  also show our $A^{[0]}$ is similar to a real symmetric matrix, so even before Theorem~\ref{theorem:eigenvalues:A0}, we knew eigenvalues had to be real.)
However, standard numerical methods to compute the eigenvalues wrongly produce complex numbers (!) with very large imaginary parts.

The reason for the numerical errors in computing the eigenvalues is that the eigenvalues of these matrices are very sensitive to small perturbations.
That phenomenal sensitivity is often characterised by the pseudospectra.
For $\epsilon>0$, the $\epsilon$-pseudospectrum is the region of the complex plane, $z \in \mathbb{C}$, where the norm of the \textit{resolvent} is large: $|| (z I - A)^{-1} || > 1/ \epsilon $.
In the $2$-norm, this is equivalent to the region where the minimum singular value, $s_{\textrm{min}}$, is small: $s_{\textrm{min}}(zI-A) < \epsilon$.

The pseudospectrum of the convection-diffusion operator is known to be significant \cite{RedTre94}, and master equations are closely related to convection-diffusion, suggesting they will also exhibit interesting pseudospectra.
Indeed, the matrices that arise in our applications of master equations to isomerizaiton exhibit an humongous pseudospectra.
They are examples of the class of \textit{twisted Toeplitz matrices and operators}, which have recently been understood to exhibit a distinctive pseudospectra, captivating more general interest \cite{TrefethenChapmanTwistedToeplitz}.

Figure~\ref{fig:A0} displays the pseudospectrum for $A^{[0]}$  and Figure~\ref{fig:A1} displays the pseudospectrum for $A^{[1]}$.
These are numerical estimates based on the algorithms underlying \textit{eigtool}.
In future work it may be possible to analytically bound the region of the complex plane where the pseudospectra is large.
For example, the pseudospectra of the convection-diffusion operator has been shown to be approximately bounded by a parabola \cite{RedTre94}, and such knowledge of this bounded region  has recently been exploited to develop effective contour integral methods based on inverse Laplace transform techniques.
Usually the idea of such methods is to choose a contour that stays away from the eigenvalues.
That works well for real symmetric matrices.
But if the operator has a significant pseudospectrum, then more is required: the contour must stay safely away from regions where the resolvent  $|| (z I - A)^{-1} || $ is large.
The figures here show some diversity in pseudospectra.
This might inspire research into a computational method that is \textit{adaptive}: instead of requiring detailed knowledge of the pseudospectrum in advance, we require computational methods that adapt the contour of integration so as to control  $|| (z I - A)^{-1} || $ to be, say, $\mathcal{O}(1)$.

\begin{figure}
  \centering
\begin{tabular}{c}
  \includegraphics[scale=0.8]{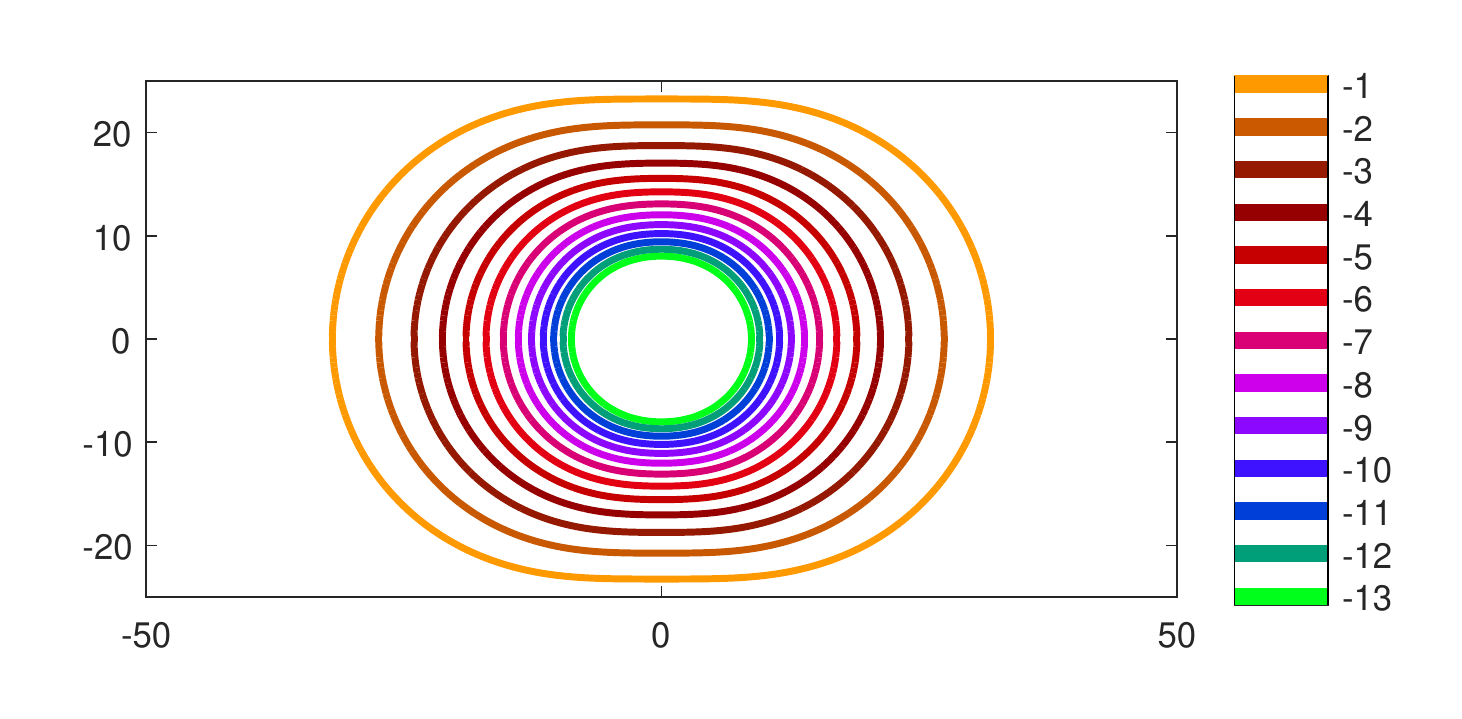}\\
  \includegraphics[scale=0.8]{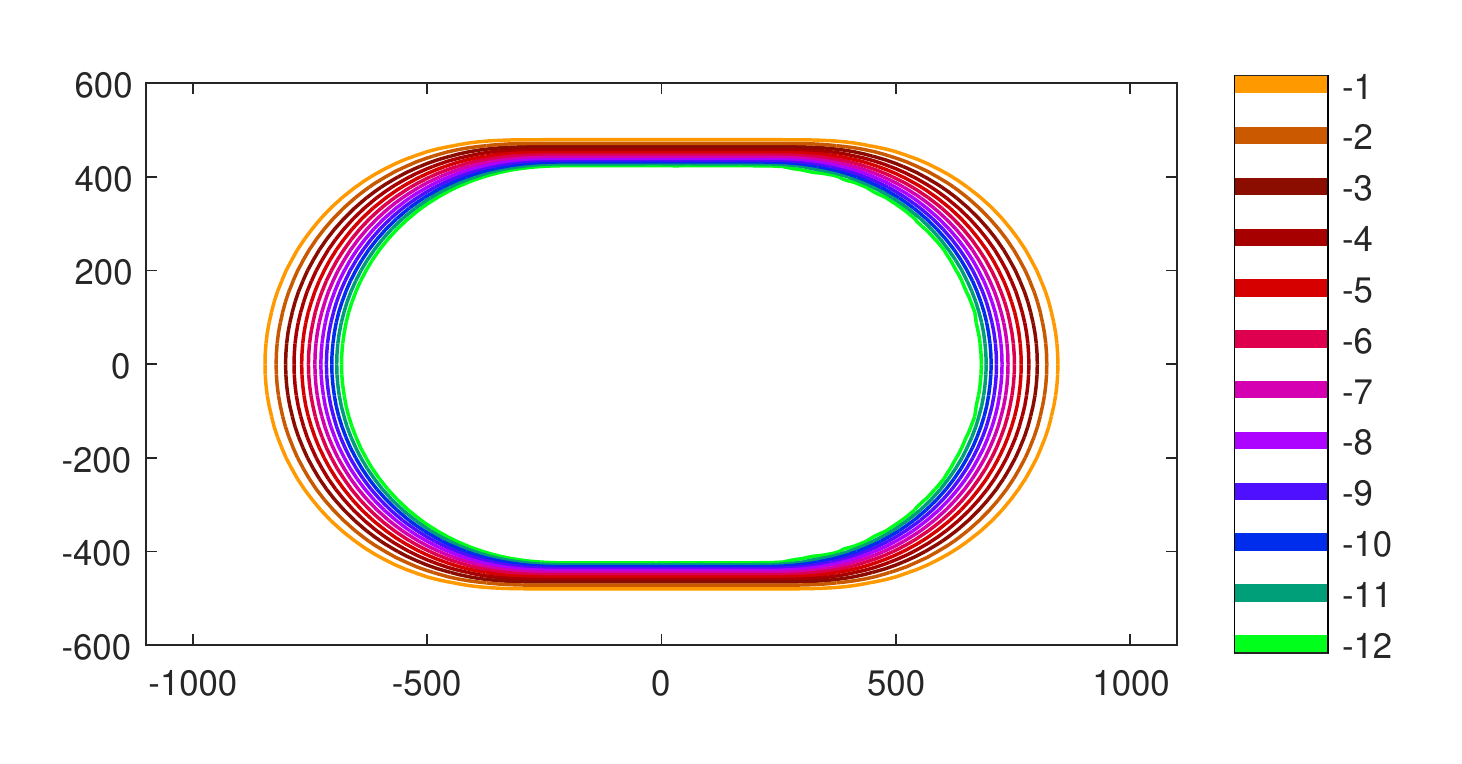}
\end{tabular}
  \caption{The `athletics track:' Pseudospectrum \cite{TreEmb05} of the $A^{[1]}$ matrix, defined in \eqref{eq:A0:A1:definitions},  as computed by Eigtool \cite{EigTool2002}.
   Top: $30 \times 30$. Bottom: $500 \times 500$.}
  \label{fig:A1}
\end{figure}

\section{Discussion}
Master equations and especially their applications will continue to occupy new directions in scientific computation for some time \cite{NewDirectionsTrefethenEssay2015}.
There is always the challenge of \textit{high dimensions}, for instance. 
Here is an incomplete list of contemporary topics where activity is growing fast.

  \subsection{Matrix functions of graph Laplacians}
  A general framework for models of biochemical kinetics has recently been elucidated in terms of graph Laplacians \cite{Gunawardena:2012aa}.
A simple example of a graph Laplacian on a line of nodes appears in \cite{StrMac14}, and, like the matrix exponential, it has been shown that a \textit{Mittag-Leffler function} \cite{MittagLefflerMonograph} of a graph Laplacian matrix is also a stochastic matrix \cite{MacNamaraFractionalEulerLimit2016}.
All of this suggests research into non-Markovian generalisations of Gillespie-like stochastic simulation algorithms allowing waiting times not exclusively drawn from an exponential distribution \cite{ShevCauchyIntegralMasterEqnPseudoSpectraCTAC2015}.

It is known that if we generalise \eqref{eq:constant:matrix:ODE} to a Caputo fractional derivative of order $0<\alpha<1$, $  \ud^{\alpha} / \ud t^{\alpha} $,  then the matrix exponential is generalised to the Mittag-Leffler function $E_{\alpha}$, so that \eqref{eq:constant:matrix:ODE} becomes
$ \ud^{\alpha} \MM{p}/ \ud t^{\alpha}  = \mathbb{A} \MM{p} $ with solution $ \bm{p}(t) = E_{\alpha}(t^{\alpha} \mathbb{A})  \bm{p}(0).$
This is assuming the coefficient matrix is constant.
However, if we allow a time-varying matrix,  $\mathbb{A}=\mathbb{A}(t)$, and generalise \eqref{eq:ODE:magnus:form} to $ \ud^{\alpha} \bm{p}/ \ud t^{\alpha}  = \mathbb{A}(t) \bm{p} $, then an important open question arises: how do we generalise the Magnus expansion of the solution?
There is certainly some work in the literature on discrete constructions of continuous-time random walks and their generalised master equations aimed at accommodating time-varying rates.
Nevertheless, the authors are not aware of a \textit{fractional generalisation of the Magnus expansion}. 
Given the  current interest in fractional processes and processes with memory, such a generalisation of the Magnus expansion would seem a timely contribution, and would presumably also suggest a fractional generalisation of the Baker--Campbell--Hausdorff formula as a special case.

\subsection{Products of matrix exponentials}
When matrices commute, a product of exponentials has an especially simple form.
 Evans,  Sturmfels \&  Uhler  recently showed how to successfully  exploit this property for master equations governing  birth-death processes \cite{UhlerSturmfelsBirthDeath2010}.
 
This computational approach has the potential for wider applications to master equations where tensor structures involving shift operators often arise. 
So let us revisit \R{eq:1.1} to find, {\em explicitly,\/} solutions (without Wilhelm Magnus and without Sophus Lie) in a way that generalises and suggests connections to products of exponentials.
To generalise \R{eq:1.1}, consider linearly independent matrices, $A$ and $B$, such that
\begin{equation}
  \label{eq:4.1}
  [A,B]=aA+bB
\end{equation}
for some $a,b\in\BB{R}$, not both zero, and the differential equation
\begin{equation}
  \label{eq:4.2}
  X'=[\alpha(t) A+\beta(t) B]X,\quad t\geq0,\qquad X(0)=I.
\end{equation}
Here $\alpha$ and $\beta$ are given scalar functions.

We wish to prove the solution of \R{eq:4.2} can be expressed in the form
\begin{equation}
  \label{eq:4.3}
  X(t)=\ee^{\rho_A(t)A}\ee^{\rho_B(t)B},
\end{equation}
where $\rho_A$ and $\rho_B$ are scalar functions obeying a certain ODE. Obviously, $\rho_A(0)=\rho_B(0)=0$.

Assume (without loss of generality) that $b\neq0$. 
Differentiating \R{eq:4.3} and substituting into \R{eq:4.2}, we have
$
   X'=\ee^{\rho_A A} (\rho_A'A+\rho_B'B)\ee^{\rho_B B}=(\alpha A+\beta B)\ee^{\rho_A A}\ee^{\rho_B B}
$
and, multiplying on the right by $\ee^{-\rho_B B}$, we have
\begin{equation}
  \label{eq:4.4}
  (\rho_A'-\alpha)A\ee^{\rho_A A}+\rho_B'\ee^{\rho_A A}B-\beta B\ee^{\rho_A A}=O.
\end{equation}

A proof by induction using \R{eq:4.1} shows 
\begin{equation}
  \label{eq:4.5}
  BA^m=(A+bI)^mB-\frac{a}{b} A[A^m-(A+bI)^m],\qquad m\in\BB{Z}_+.
\end{equation}
Consequently,
$
  B\ee^{\rho_A A} = \sum_{m=0}^\infty \frac{\rho_A^m}{m!} BA^m =
  $
  \[
  \sum_{m=0}^\infty \frac{\rho_A^m}{m!} (A+bI)^m B -\frac{a}{b}A\sum_{m=0}^\infty \frac{\rho_A^m}{m!} [A^m-(A+bI)^m]
   = \ee^{b\rho_A}\ee^{tA}B-\frac{a}{b} (1-\ee^{b\rho_A})A\ee^{\rho_A A}.
\]
Now substitute into \R{eq:4.4},
$
  (\rho_A'-\alpha)A\ee^{\rho_A A}+\rho_B'\ee^{\rho_A A}B-\beta\ee^{b\rho_A}\ee^{tA}B+\frac{a}{b}\beta (1-\ee^{b\rho_A})A\ee^{tA}.
$
Separating between $A\ee^{\rho_A A}$ and $\ee^{\rho_A A}B$ above, we obtain two ODEs for $\rho_A$ and $\rho_B$,
\begin{eqnarray}
  \label{eq:4.6}
  \rho_A'&=&\alpha-\frac{a}{b} \beta (1-\ee^{b\rho_A}),\qquad \rho_A(0)=0,\\
  \label{eq:4.7}
  \rho_B'&=&\beta \ee^{b\rho_A},\qquad \rho_B(0)=0,
\end{eqnarray}
reducing the computation of $\rho_A$ to a scalar ODE and of $\rho_B$ to quadrature.

Specialising to master equations, $\alpha\equiv1$, $\beta=f$, $a=0$ and $b=-2$, so \R{eq:4.6} becomes
$
  \rho_A(t)=t \textrm{ and } \rho_B(t)=\int_0^t \ee^{-2\tau}f(\tau)\D \tau.
$
Putting \R{eq:4.7} in \R{eq:4.6}, we obtain
$
  \rho_A'=\alpha-\frac{a}{b}\beta+\frac{a}{b}\rho_B'.
$
Multiplication by $b$ and integration implies the integral
$
  b\rho_A(t)-a\sigma(t)=b\int_0^t \alpha(\tau)\D\tau -a\int_0^t \beta(\tau)\D \tau.
$

\textit{Can all this be (further) generalised, beyond two exponentials?}
We now suggest the answer to this question is affirmative although applications form the subject of ongoing research.
Indeed what we have done thus far is to exemplify {\em precisely\/} the Wei--Norman approach of expressing the solution of a linear ODE using \textit{canonical coordinates of the second kind} \cite{wei64ogr}. 
Specifically, let $A:\BB{R}_+\rightarrow\GG{g}$, where $\GG{g}$ is a Lie algebra, $\dim\GG{g}=d$, and consider the ODE
\begin{equation}
  \label{eq:4.9}
  X'=A(t)X,\quad t\geq0,\qquad X(0)=I.
\end{equation}
Let $\mathcal{P}=\{P_1,P_2,\ldots,P_d\}$ be a basis of $\GG{g}$. 
Wei \& Norman \cite{wei64ogr} prove that for sufficiently small $t>0$ there exist functions $g_1,g_2,\ldots,g_d$ such that
\begin{equation}
  \label{eq:4.10}
  X(t)=\ee^{g_1(t)P_1}\ee^{g_2(t)P_2}\cdots\ee^{g_d(t)P_d}.
\end{equation}
This is the situation we have in \R{eq:1.1} or, with greater generality, in \R{eq:4.2}: $P_1=A$, $P_2=B$ and, because of \R{eq:4.1}, the dimension of the free Lie algebra spanned by $A$ and $B$ is  $d=2$. 
Interestingly enough,  this example does not feature in \cite{wei64ogr}.

Coordinates of the second kind have been used extensively in the theory of  Lie-group integrators  \cite{iserles00lgm} where it always followed an organising principle that also shows promise for master equations. 
Specifically, the assumption was -- unlike our simple $d=2$ example -- that $d$ is large (e.g.\ that $\GG{g}$ is the special orthogonal group of matrices $\CC{SO}(n)$, say, or the special linear group of matrices $\CC{SL}(n)$) and the basis $\mathcal{P}$ selected so that it is easy to evaluate the exponentials $\exp (g_k P_k)$ (e.g., using root space decomposition) \cite{celledoni01mam}. 

  \subsection{Pseudospectra of master equations}
  This is a subject worthy of more attention.
       For example, we have shown here that even simple isomerisation models exhibit a highly non-trivial pseudospectra.
       We conjecture  that \textit{Michaelis--Menten enzyme kinetics} and a whole host of other important models in biology also exhibit significant pseudospectra \cite{ShevCauchyIntegralMasterEqnPseudoSpectraCTAC2015,NewDirectionsTrefethenEssay2015}.
       In the usual model of Michaelis--Menten kinetics, a catalytic enzyme $E$ reversibly forms a complex intermediate $C$  with a substrate $S$, that is eventually irreversibly converted to a product $P$, viz.  $S+E \leftrightarrow C \rightarrow P+E$. 
       There is a need for visualisations of the pseudospectrum of such Michaelis-Menten kinetics, for example.
        Another open question is how the pseudospectrum of the usual model compares to the pseudospectrum of a more reasonable model suggested by Gunawardena to repent for the ``Original Thermodynamic Sin'' of including the irreversible reaction $C \rightarrow P+E$ \cite{GunawardenaOriginalSin}.

        As a demonstration of this topic going far beyond merely the isomerisation examples that we have studied here, we have also computed here in Figure~\ref{fig:TASEP} the pseudospectrum of the \textit{totally asymmetric exclusion process} (TASEP) \cite[Figure 9]{CorwinKPZNoticesAMS2016}.
        If all that is observed in the picture of the pseudospectrum is merely some `$\epsilon-$balls,' centred around each eigenvalue, and well-separated, then the situation is not interesting.
        For that is simply the picture we would expect for a well-behaved real symmetric matrix anyway. 
        To be interesting, more complex behaviour is required. 
        It is too early to tell for the TASEP, but our preliminary  numerical picture here in Figure~\ref{fig:TASEP} suggests it will turn out to be worthwhile pursuing.
        The figure depicts the case with six particles and we can already discern the beginnings of some interesting interactions emerging.
        Such examples of TASEP models have found applications to single molecule studies of RNA polymerase and protein synthesis.
        More generally exclusion processes have witnessed a renaissance of mathematical interest, partly in relation to exactly integrable probabilistic systems, the Kardar--Parisi--Zhang (KPZ) universality class, and the KPZ  stochastic partial differential equation   \cite{CorwinICM2014, HairerICM}.

    \textit{Random Matrix Theory} \cite{ActaNumericaRandomMatrixTheory2005}  connects to master equations.
     For example,  an important limiting distribution associated with the TASEP master equation is the famous Tracy--Widom distribution for the biggest eigenvalue of a large, random Hermitian matrix  \cite{CorwinICM2014}.
     Although less in the sense of the chemical master equation  (at least so far but that could change) and more in the physicists' sense of Wigner and Freeman Dyson, random matrix theory is also playing a role in recent studies of random graph Laplacians.
     The resulting distributions are very similar to the standard Gaussian ensembles but the special algebraic properties of graph Laplacians do lead to peculiar discrepancies that persist for large matrix dimension $N$   \cite{CarstenTimmRandomMatrixTheoryMasterEquation2009}.
    Interestingly, the \textit{Matrix-Tree Theorem}, which gives a formula for the stationary distribution (and confirmation of positivity) of such master equations in terms of sums of positive off-diagonal entries, seems yet to be exploited in this random matrix context.

\begin{figure}
  \centering
  \includegraphics[scale=0.8]{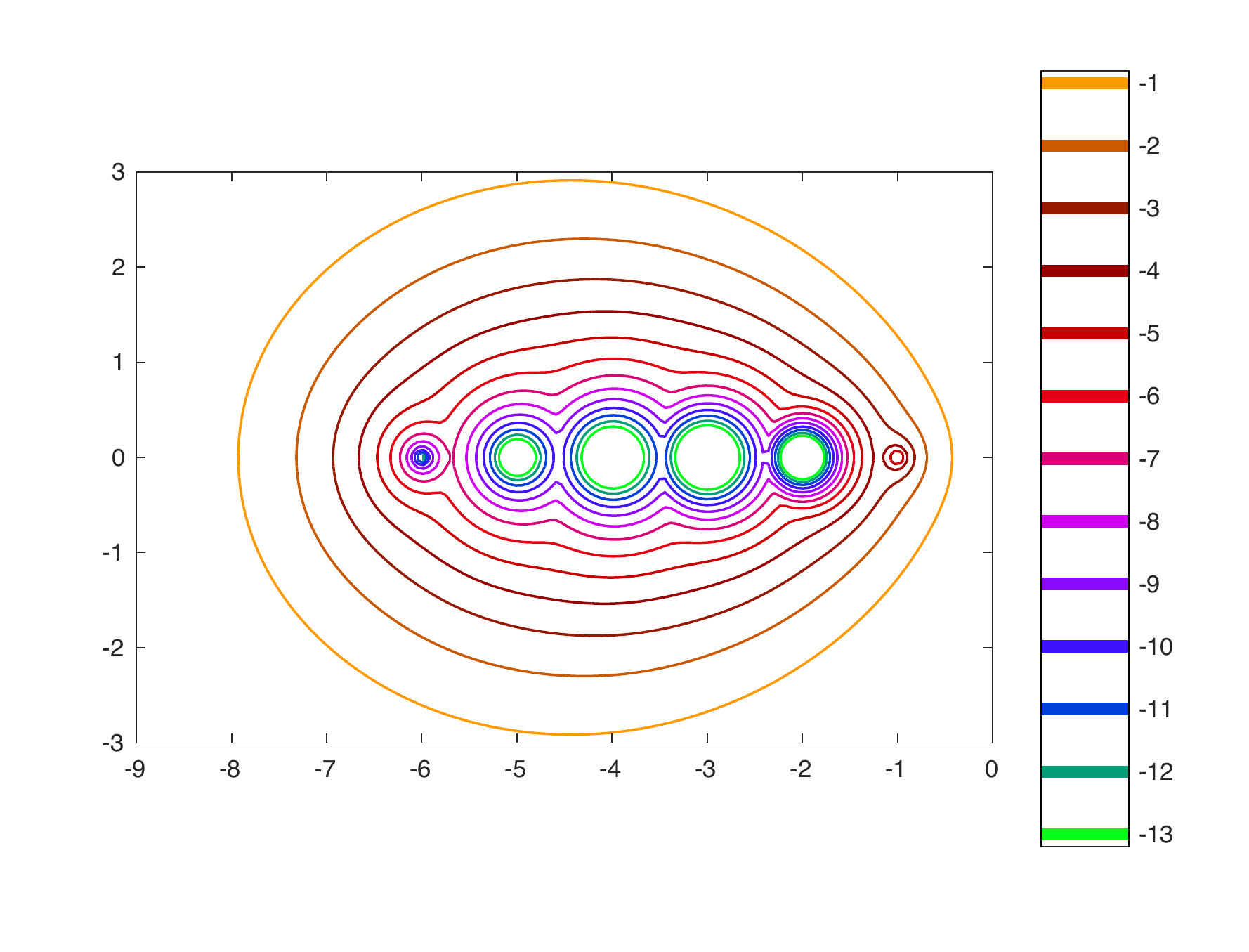}
  \caption{The `seed pod:' Pseudospectrum \cite{TreEmb05} of a $1513 \times 1513$ finite section of the singly infinite matrix associated with a totally asymmetric exclusion process (TASEP)  with 6 particles beginning in a `step' initial configuration \cite[Figure 9, q=0]{CorwinKPZNoticesAMS2016},  as computed by Eigtool \cite{EigTool2002}.  }
  \label{fig:TASEP}
\end{figure}

  \subsection{Magnus expansions and Kurtz's random time-change representation}
Denote the forward rate by $\alpha_f \left( \bm{x}(s), s \right)=c_1(s) n_1 $ and the backward rate by $\alpha_b \left( \bm{x}(s), s \right) =  c_2(s) n_2$.
Here $n_1$ and $n_2$ are the number of molecules of $S_1$ and $S_2$, respectively.
The Kurtz random time-change representation \cite{Kur80} of the sample paths  corresponding to our master equation  \R{eq:1.1} with initial state $\bm{x}(0)$ is
\begin{equation*}
\bm{x}(t) = \bm{x}(0)  +
\left(
\begin{array}{c}
-1\\
+1
\end{array}
\right) Y_1\left( \int_0^t \alpha_f \left( \bm{x}(s), s \right) \ud s \right)
 +
\left(
\begin{array}{c}
+1\\
-1
\end{array}
\right) Y_2 \left( \int_0^t \alpha_b \left( \bm{x}(s), s \right) \ud s \right) .
\end{equation*}
At absolute time $t$, this stochastic equation has  two \textit{internal time frames}: $T_j=\int_0^t \alpha_j(\bm{x}(s),s)ds$,  $j=1,2$.
Here, $Y_1$ and $Y_2$ are independent, unit-rate Poisson processes but dependencies arise through the rates in these internal time-frames.
  Thus Kurtz and Magnus offer  two different representations of the same solution, when rates  are time-varying.
Although much work has appeared on each representation separately, there has been almost no work exploring connections.
Such connections would perhaps  allow probabilistic interpretations of the Magnus expansion.

More generally \textit{time-varying rates} are one way to model \textit{extrinsic noise}, so methods  that can accommodate time-varying rates, such as Magnus expansions described here, may find wider applications \cite{HellanderLotstedtMacNamaraSpatialExtrinsic15, HilfingerPaulsson2011}.
Exploring the robustness of master equations to perturbations, including time-varying perturbations, might bring together methods from Magnus-like approaches, pseudospectral studies, and perhaps even stochastic operator approaches \cite{ActaNumericaRandomMatrixTheory2005}.

       Kurtz's representation has also inspired multi-level Monte Carlo (MLMC) methods to be adapted from the setting of SDEs to the setting of master equations, and in turn this has led to MLMC methods for estimating the \textit{sensitivity} \cite{AndersonSensitivity2012}.
  It will be interesting to see if \textit{adjoint methods} for sensitivity estimates in the setting of continuous SDEs such as the methods for which Giles and Glasserman won \textit{Risk `Quant-of-the-Year'} \cite{GilGla06} are likewise adaptable to the discrete setting of master equations \cite{KormannMacNamara2016}.

  \subsection{Preserving positivity}
  Moler and Van Loan discuss more than nineteen dubious ways for computing the \textit{matrix exponential} \cite{Moler03}.
  When such methods are applied to the important class of graph Laplacian matrices --- as arise in all master equations and Markov processes, and for which the matrix exponential is provably nonnegative and indeed a stochastic matrix --- a fundamental question is: do these numerical methods preserve nonnegativity?
  For example, does    \texttt{MATLAB}'s \texttt{expm} function preserve positivity when applied to a graph Laplacian matrix?
   This question seems especially ripe for research in relation to Krylov-like approximations, Pad\'e-like approximations with scaling and squaring, and recent methods of Al-Mohy and Higham (which are currently the basis of \texttt{expm}  in \texttt{MATLAB})  \cite{HighamExponential09,MohHig11}.
   
     We found the complete Magnus expansion for our isomerisation model.
      Being the full and exact Magnus expansion, it respects the original properties of the system, such as maintaining positivity. 
     Numerical methods in other contexts are often derived by truncation of the Magnus expansion, to a certain prescribed order. 
     In general,  \textit{truncation} of the Magnus expansion does not result in the same properties as a graph Laplacian, so positivity is no longer guaranteed.
     (Although if we are willing to settle for second-order accuracy, then it is possible to truncate so as to maintain these desirable properties.)
     The issue is that the commutator of two graph Laplacians is not in general a graph Laplacian; it may have negative off-diagonal entries.
     This observation is motivating ongoing research whose roots are in geometric numerical integration --- a subject usually concerned with maintaining equalities --- to allow the preservation of \textit{inequalities}, such as preserving positivity.

  More generally it has been known for a long time in the context of ODEs that standard numerical methods such as Runge--Kutta methods, usually do not preserve positivity unless they are of  first order accuracy \cite{bolley78cpd}.
  This also presents a contemporary challenge for Monte Carlo simulation of the sample paths of  master equations:  the widely used \textit{tau-leap methods} and other analogues of the Euler method or of the Euler--Maruyama method, cannot be guaranteed to preserve positivity.
  This challenge is motivating much current research appearing on approximations that are able to maintain positivity in these settings, as exemplified in the Kolmogorov Lecture at the most recent World Congress In Probability and Statistics \cite{WilliamsKolmogorovLecture2016}.

\section{Conclusions}
\label{sec:conclusions}

Pafnuty Chebyshev was an academic parent of Markov and today the world has come full circle with Chebyshev polynomials being a useful basis for numerical solvers of Markovian master equations in the quantum world \cite{QuantumChemistrySoftware}.
Here the adjective `master' is not used in the sense of an overlord; rather it is in the sense of an ensemble averaging principle that emerges at larger scales from the collective behaviour of the mob of microscopic particles, each following their own random walk.
Edelman and Kostlan  take such a walk on ``the road from Kac's matrix to Kac's polynomials,'' and our own matrix examples $ A^{[0]}$ and $  A^{[1]}$ of  \eqref{eq:A0:A1:definitions} also lie at the end of that road, being almost the ``\textit{Kac matrix}'' (as named by Olga Taussky and John Todd)  and ``anti-Kac matrix''  \cite{EdelmanRoadfromKactoKac1994}.
These matrices have served us well as wonderful running examples to illustrate new directions in master equation research. 
Kac did not foresee our applications to isomerisation, nor the way those isomerisation master equations are so naturally amenable to Magnus expansions.
Similarly, these and other applications that we have surveyed, such as the inchoate subject of the pseudospectra of master equations, no doubt have a bright future that we have yet to fully imagine.

\section*{Acknowledgments}
This research and Shev MacNamara have been partially supported by a David G. Crighton Fellowship to DAMTP, Cambridge.
Arieh Iserles presented some of the results of this article in a workshop at the Oberwolfach (\url{https://na.math.kit.edu/marlis/research/meetings/16-oberwolfach/Iserles.pdf}) and thanks participants for interest in the geometry of master equations.
He also acknowledges a fruitful discussion with Nick Trefethen on the pseudospectrum of $A^{[1]}$.

\bibliographystyle{siamplain}
\bibliography{refs}

\begin{thebibliography}{10}

\bibitem{HighamExponential09}
{\sc A.~H. Al-Mohy and N.~J. Higham}, {\em A new scaling and squaring algorithm
  for the matrix exponential}, SIAM J. Matrix Anal. Appl.,  (2009),
  pp.~970--989.

\bibitem{MohHig11}
{\sc A.~H. Al-Mohy and N.~J. Higham}, {\em Computing the action of the matrix
  exponential, with an application to exponential integrators}, SIAM J. Sci.
  Comp., 33 (2011), pp.~488--511, \url{https://doi.org/10.1137/100788860}.

\bibitem{AndersonSensitivity2012}
{\sc D.~Anderson}, {\em An efficient finite difference method for parameter
  sensitivities of continuous time markov chains}, SIAM Journal on Numerical
  Analysis, 50 (2012), pp.~2237--2258.

\bibitem{QuantumChemistrySoftware}
{\sc J.~R. Barker, T.~L. Nguyen, J.~F. Stanton, M.~C. C.~Aieta, F.~Gabas,
  T.~J.~D. Kumar, C.~G.~L. Li, L.~L. Lohr, A.~Maranzana, N.~F. Ortiz, J.~M.
  Preses, and P.~J. Stimac}, {\em Multiwell-2016 software suite}, tech. report,
  University of Michigan, Ann Arbor, Michigan, USA, 2016,
  \url{http://clasp-research.engin.umich.edu/multiwell/}.

\bibitem{bolley78cpd}
{\sc C.~Bolley and M.~Crouzeix}, {\em Conservation de la positivit\'e lors de
  la discr\'etisation des probl\`emes d'\'evolution paraboliques}, RAIRO Anal.
  Num\'er., 12 (1978), pp.~237--245, iv.

\bibitem{celledoni01mam}
{\sc E.~Celledoni and A.~Iserles}, {\em Methods for the approximation of the
  matrix exponential in a {L}ie-algebraic setting}, IMA J. Numer. Anal., 21
  (2001), pp.~463--488, \url{https://doi.org/10.1093/imanum/21.2.463},
  \url{http://dx.doi.org/10.1093/imanum/21.2.463}.

\bibitem{CorwinICM2014}
{\sc I.~Corwin}, {\em {Macdonald processes, quantum integrable systems and the
  Kardar--Parisi--Zhang universality class}}, in Proceedings of the
  International Congress of Mathematicians, arXiv:1403.6877, 2014.

\bibitem{CorwinKPZNoticesAMS2016}
{\sc I.~Corwin}, {\em {Kardar--Parisi--Zhang Universality}}, Notices of the
  AMS, 63 (2016).

\bibitem{PetzoldCloud}
{\sc B.~Drawert, M.~Trogdon, S.~Toor, L.~Petzold, and A.~Hellander}, {\em
  Molns: A cloud platform for interactive, reproducible, and scalable spatial
  stochastic computational experiments in systems biology using pyurdme}, SIAM
  Journal on Scientific Computing, 38 (2016), pp.~C179--C202,
  \url{https://doi.org/10.1137/15M1014784}.

\bibitem{EdelmanRoadfromKactoKac1994}
{\sc A.~Edelman and E.~Kostlan}, {\em {The road from Kac's matrix to Kac's
  random polynomials}}, tech. report, University of California, Berkeley, 1994.

\bibitem{ActaNumericaRandomMatrixTheory2005}
{\sc A.~Edelman and N.~R. Rao}, {\em Random matrix theory}, Acta Numerica,
  (2005), pp.~1--65.

\bibitem{UhlerSturmfelsBirthDeath2010}
{\sc S.~N. Evans, B.~Sturmfels, and C.~Uhler}, {\em Commuting birth-and-death
  processes}, The Annals of Applied Probability, 20 (2010), pp.~238--266.

\bibitem{GilGla06}
{\sc M.~Giles and P.~Glasserman}, {\em Smoking adjoints: fast {M}onte {C}arlo
  {G}reeks}, Risk,  (2006), p.~88.

\bibitem{Gillespieisomerization2002}
{\sc D.~T. Gillespie}, {\em The chemical {L}angevin and {F}okker--{P}lanck
  equations for the reversible isomerization reaction}, The Journal of Physical
  Chemistry A, 106 (2002), pp.~5063--5071,
  \url{https://doi.org/10.1021/jp0128832}.

\bibitem{MittagLefflerMonograph}
{\sc R.~Gorenflo, A.~Kilbas, F.~Mainardi, and S.~Rogosin}, {\em Mittag-Leffler
  Functions, Related Topics and Applications}, Springer, 2014.

\bibitem{Gunawardena:2012aa}
{\sc J.~Gunawardena}, {\em A linear framework for time-scale separation in
  nonlinear biochemical systems}, PLoS One, 7 (2012), p.~e36321,
  \url{https://doi.org/10.1371/journal.pone.0036321}.

\bibitem{GunawardenaOriginalSin}
{\sc J.~Gunawardena}, {\em {Time-scale separation: Michaelis and Menten's old
  idea, still bearing fruit}}, FEBS J., 281 (2014), pp.~473--488.

\bibitem{HairerICM}
{\sc M.~Hairer}, {\em Singular stochastic {PDEs}}, Proceedings of the
  International Congress of Mathematicians,  (2014).

\bibitem{HellanderLotstedtMacNamaraSpatialExtrinsic15}
{\sc A.~Hellander, J.~Klosa, P.~L\"otstedt, and S.~MacNamara}, {\em Robustness
  analysis of spatiotemporal models in the presence of extrinsic fluctuations},
  submitted, SIAM Journal on Applied Mathematics, arXiv:1610.01323,  (2015).

\bibitem{HighamD2008}
{\sc D.~J. Higham}, {\em Modeling and simulating chemical reactions}, SIAM
  Review, 50 (2008), pp.~347--368, \url{https://doi.org/10.1137/060666457},
  \url{http://dx.doi.org/10.1137/060666457}.

\bibitem{HilfingerPaulsson2011}
{\sc A.~Hilfinger and J.~Paulsson}, {\em Separating intrinsic from extrinsic
  fluctuations in dynamic biological systems}, Proc. Acad. Natl. Sci., 109
  (2011), pp.~12167--72, \url{https://doi.org/10.1073/pnas.1018832108}.

\bibitem{Top5MarkovApplications}
{\sc P.~V. Hilgers and A.~N. Langville}, {\em {The five greatest applications
  of Markov chains}}, in Proceedings of the Markov Anniversary Meeting, Boston
  Press, Boston, MA., 2006.

\bibitem{HochbruckLubich03}
{\sc M.~Hochbruck and C.~Lubich}, {\em On {M}agnus integrators for
  time-dependent {S}chr{\"o}dinger equations}, SIAM J. Numer. Anal., 41 (2003),
  pp.~945--963, \url{https://doi.org/10.1137/S0036142902403875}.

\bibitem{iserles00lgm}
{\sc A.~Iserles, H.~Z. Munthe-Kaas, S.~P. N{\o}rsett, and A.~Zanna}, {\em
  Lie-group methods}, Acta Numer., 9 (2000), pp.~215--365,
  \url{https://doi.org/10.1017/S0962492900002154},
  \url{http://dx.doi.org/10.1017/S0962492900002154}.

\bibitem{Jahnke2007}
{\sc T.~Jahnke and W.~Huisinga}, {\em Solving the chemical master equation for
  monomolecular reaction systems analytically}, Journal of Mathematical
  Biology, 54 (2007), pp.~1--26,
  \url{https://doi.org/10.1007/s00285-006-0034-x},
  \url{http://www.scopus.com/inward/record.url?eid=2-s2.0-33845629747&partnerID=40&md5=c947b5e7b11c3810334b2232d40169e6}.
\newblock cited By 97.

\bibitem{Kac1957bookMasterEquation}
{\sc M.~Kac}, {\em Probability and Related Topics in Physical Sciences, Summer
  Seminar in Applied Mathematics, Boulder, Colorado}, American Mathematical
  Society, 1957.

\bibitem{KormannMacNamara2016}
{\sc K.~Kormann and S.~MacNamara}, {\em Error control for exponential
  integration of the master equation}, arXiv:1610.03232,  (2016).

\bibitem{Kur80}
{\sc T.~Kurtz}, {\em Representations of {M}arkov processes as multiparameter
  time changes}, Ann. Probab., 8 (1980), pp.~682--715.

\bibitem{WilliamsKolmogorovLecture2016}
{\sc S.~C. Leite and R.~J. Williams}, {\em {A constrained Langevin
  approximation for chemical reaction networks}}, Kolmogorov Lecture, Ninth
  World Congress In Probability and Statistics, Toronto,  (2016).

\bibitem{ShevCauchyIntegralMasterEqnPseudoSpectraCTAC2015}
{\sc S.~Macnamara}, {\em Cauchy integrals for computational solutions of master
  equations}, ANZIAM Journal, 56 (2015), pp.~32--51,
  \url{https://doi.org/10.21914/anziamj.v56i0.9345}.

\bibitem{MacBur08}
{\sc S.~MacNamara, K.~Burrage, and R.~Sidje}, {\em Multiscale modeling of
  chemical kinetics via the master equation}, SIAM Multiscale Model. \& Sim., 6
  (2008), pp.~1146--1168.

\bibitem{MacNamaraFractionalEulerLimit2016}
{\sc S.~MacNamara, B.~I. Henry, and W.~McLean}, {\em Fractional {E}uler limits
  and their applications}, SIAM Journal on Applied Mathematics,  (2016).

\bibitem{NewDirectionsTrefethenEssay2015}
{\sc S.~MacNamara and G.~Strang}, {\em Master equations in `{E}ssays on {N}ew
  {D}irections in {N}umerical {C}omputation'}, 2015,
  \url{http://tobydriscoll.net/newdirections2015/}.

\bibitem{Magnus54}
{\sc W.~Magnus}, {\em On the exponential solution of differential equations for
  a linear operator}, Comm. Pure Appl. Math., 7 (1954), pp.~649--673.

\bibitem{Moler03}
{\sc C.~Moler and C.~V. Loan}, {\em Nineteen dubious ways to compute the
  exponential of a matrix, twenty-five years later}, SIAM Rev., 45 (2003),
  pp.~3--49, \url{https://doi.org/10.1137/S00361445024180}.

\bibitem{munthekass01gpd}
{\sc H.~Z. Munthe-Kaas, G.~R.~W. Quispel, and A.~Zanna}, {\em Generalized polar
  decompositions on {L}ie groups with involutive automorphisms}, Found. Comput.
  Math., 1 (2001), pp.~297--324, \url{https://doi.org/10.1007/s102080010012},
  \url{http://dx.doi.org/10.1007/s102080010012}.

\bibitem{PavliotisStuart2008book}
{\sc G.~A. Pavliotis and A.~Stuart}, {\em Multiscale Methods: Averaging and
  Homogenization}, Springer, 2008.

\bibitem{RedTre94}
{\sc S.~C. Reddy and L.~N. Trefethen}, {\em Pseudospectra of the
  convection-diffusion operator}, SIAM J. Appl. Math,  (1994).

\bibitem{StrMac14}
{\sc G.~Strang and S.~MacNamara}, {\em Functions of difference matrices are
  {T}oeplitz plus {H}ankel}, SIAM Review, 56 (2014), pp.~525--546,
  \url{https://doi.org/10.1137/120897572}.

\bibitem{CarstenTimmRandomMatrixTheoryMasterEquation2009}
{\sc C.~Timm}, {\em Random transition-rate matrices for the master equation},
  Phys. Rev. E, 80 (2009), p.~021140.

\bibitem{TrefethenChapmanTwistedToeplitz}
{\sc L.~N. Trefethen and S.~J. Chapman}, {\em Wave packet pseudomodes of
  twisted {T}oeplitz matrices}, Comm. Pure Appl. Math., 57 (2004),
  pp.~1233--1264, \url{https://doi.org/10.1002/cpa.20034},
  \url{http://dx.doi.org/10.1002/cpa.20034}.

\bibitem{TreEmb05}
{\sc L.~N. Trefethen and M.~Embree}, {\em Spectra and {P}seudospectra: {T}he
  {B}ehavior of {N}onnormal {M}atrices and {O}perators}, Princeton University
  Press, 2005.

\bibitem{MarkusWeber2016reviewFeynmanKacMasterEqn}
{\sc M.~F. Weber and E.~Frey}, {\em Master equations and the theory of
  stochastic path integrals}, arXiv:1609.02849v1,  (2016).

\bibitem{wei64ogr}
{\sc J.~Wei and E.~Norman}, {\em On global representations of the solutions of
  linear differential equations as a product of exponentials}, Proc. Amer.
  Math. Soc., 15 (1964), pp.~327--334.

\bibitem{EigTool2002}
{\sc T.~G. Wright}, {\em Eigtool}, 2002,
  \url{http://www.comlab.ox.ac.uk/pseudospectra/eigtool/}.

\end{thebibliography}
\end{document}